\documentclass[11pt]{amsart}
\usepackage{lmodern}

\usepackage{amsmath,amsthm,amssymb,amsfonts}
\usepackage[utf8]{inputenc}
\DeclareMathAlphabet\mathbfcal{OMS}{cmsy}{b}{n}
\usepackage{xcolor}
\usepackage{mathtools}
\usepackage{tikz-cd,tikz}
\usepackage[
    colorlinks=true,
    linkcolor=black,
    urlcolor=cyan,
    citecolor=magenta
]{hyperref}

\hypersetup{pdftitle={On the Feyzbakhsh--Thomas programme for Fano 3-folds}}

\theoremstyle{definition}
\newtheorem{definition}{Definition}[section]

\newtheorem{theorem}[definition]{Theorem}

\newtheorem{proposition}[definition]{Proposition}

\newtheorem{corollary}[definition]{Corollary}

\newtheorem{lemma}[definition]{Lemma}

\newtheorem{remark}[definition]{Remark}

\newtheorem{thm}{Theorem}

\newcommand{\heart}{\ensuremath\heartsuit}

\DeclareMathOperator{\id}{id}

\title[On the Feyzbakhsh-Thomas programme for Fano $3$-folds]{On the Feyzbakhsh-Thomas programme for Fano $3$-folds}
\date{\today}

\author[I. Karpov]{Ivan Karpov (with an Appendix joint with Miguel Moreira)}

\begin{document}

\begin{abstract}
Let $X$ be a Fano $3$-fold with even canonical class which satisfies the generalized Bogomolov-Gieseker inequality, such as $\mathbb P^3$. We express Donaldson-Thomas invariants counting Gieseker semistable sheaves of rank $r$, where $r > 0$, on $X$ in terms of those counting sheaves of rank $0$ and pure dimension $2$. This implements an analogue of the programme initiated by S. Feyzbakhsh and R. Thomas in the case of Calabi-Yau varieties.

The methods include $K$-theoretic Donaldson-Thomas theory and, quite unexpectedly, the use of certain combinatorial properties of vertex algebras.
\end{abstract}

\maketitle

\tableofcontents

\section{Introduction}\label{intro}
\subsection{} The enumerative geometry of Calabi-Yau threefolds constitutes a well-established area. These invariants are \textit{numbers} associated to each moduli space of sheaves with a given $K$-theory class (or given Chern character) on a  Calabi-Yau 3-fold: the numbers are, in a suitable sense, Euler characteristics of these moduli spaces.

One of the recent advances in the area is the work of Feyzbakhsh and Thomas which formally reduces  calculations of invariants for rank $r$-sheaves to those of rank $0$ (\cite{FT}; and to $\operatorname{rank} = 1$: \cite{FT1}) solving a long-standing question. 

The goal of the present work is the development of an analogous programme for \textit{Fano} $3$-folds following the idea of Dominic Joyce, see \cite{JoycePres}. This task has a nature different from the situation of Calabi-Yau $3$-folds: the enumerative invariants now depend on some additional data and reflect a finer \textit{intersection-theoretic} structure of the moduli spaces.

\subsection{} Let us give some sense of the problem under consideration. Let $X$ be a smooth projective Fano $3$-fold. Fix a $K$-theory class $\alpha$ of dimension $>0$, and consider the moduli space $\mathcal M(\alpha)$ of Gieseker semistable sheaves in class $\alpha$.

Suppose for now, for the sake of demonstration, that all semistable sheaves in $\mathcal M(\alpha)$ are stable, and that there exists a universal sheaf $\mathcal U$ over $\mathcal M(\alpha) \times X.$ \footnote{In general, the first assumption will be  removed by replacing $[\mathcal M]^{vir}$ with the so-called Joyce's class, and the second one -- by working with the so-called $\operatorname{weight} 0$-descendants. We discuss this below.} 

Since $X$ is Fano, $\mathcal M(\alpha)$ has a perfect obstruction theory, and, since semistable objects are stable, there exists a \textit{Behrend-Fantechi virtual fundamental class} $[\mathcal M(\alpha)]^{vir}$ in $H_{2d}(\mathcal M(\alpha))$ for some $d \geq 0$ which is known as the virtual dimension of $\mathcal M(\alpha)$, see \cite{BF}.

In the case of Calabi-Yau varieties, the virtual dimension is always zero: the enumerative invariants are, as a consequence, numbers associated directly to the data of $\mathcal M(\alpha)$. For Fano varieties, however, to get a number from the virtual class, one has to  put additionally a cohomology class $\tau$ on $\mathcal M(\alpha)$, and to integrate it against $[\mathcal M(\alpha)]^{vir}$.

A natural choice is the classes of the form \begin{equation}\label{intd} \tau = \pi_{1*}(\operatorname{ch}_i(\mathcal U) \cup \pi_2^*(\kappa))\end{equation} known as descendant classes. Here $\kappa \in H^*(X)$, $i \in \mathbb Z$, and $\pi_j$ are the projections from $\mathcal M \times X$ to its two respective factors.

The enumerative invariants of $X$ are integrals of the form \begin{equation}\label{desc1}I(\tau) = \int_{[\mathcal M]^{vir}}\tau\end{equation} where $\tau$ is a descendant from Equation (\ref{intd}), or a product thereof.

In the present work, we prove that, for an arbitrary class $\alpha$ and arbitrary descendant $\tau$ with an \textit{algebraic insertion $\kappa$}, all integrals of the form (\ref{desc1}) are formally determined by those where $\psi$ is a descendant, and $\beta$ is a rank $0$, pure dimension $2$ class.

This constitutes \textit{a rank reduction algorithm}, the main technique of which will pass through \textit{K-theoretic Donaldson-Thomas theory}. The precise formulation will be given in Theorem \ref{mainthm}.

\subsection{}
We now give some more details. 

First of all, for any derived Artin stack $\mathfrak X,$ the notation $K(\mathfrak X)$ in this paper always stands for the direct sum of odd and even Blanc's topological $K$-theory with complex coefficients, see \cite{Blanc}, and \cite[Appendix A]{KM}. If $\mathfrak X$ is a scheme, it is known that $K(\mathfrak X)\simeq H^*(\mathfrak X, \mathbb C),$ and the isomorphism is established via the Chern character. For other generalities about this invariant, we refer to \cite[Appendix A]{KM}.

Let $X$ be a smooth, projective Fano $3$-fold over $\mathbb C$, which satisfies the generalized Bogomolov-Gieseker inequality of \cite{BMT}, and such that $K_X$ is divisible by $2$. We assume that $X$ is polarized by a divisor $H$ so that $K_X$ is an integer multiple of $H$. 

\begin{remark} We would like to assume also, for technical reasons, that $X$ is \textit{of class D}. Fortunately, however, it turns out that this condition \textit{always} holds for Fano varieties by a result of Voineagu, see the explanation in Section \ref{Kthsec}.
\end{remark}

For example, $X = \mathbb P^3$ satisfies all of our assumptions, by the results of~\cite{M}.

Suppose that $\alpha  \in K(X)$ is a $K$-theory class. Let $\underline{\mathcal M}(\alpha)$ be the moduli stack of  Gieseker semistable sheaves on $X$ of class $\alpha$; let $\mathcal M(\alpha):= \underline{\mathcal M}(\alpha)^{rig}$ be its rigidified version.

In his foundational work~\cite{Joy}, D. Joyce has defined the ``virtual fundamental class'' as an element $[\mathcal M]^{vir} \in H_*(\mathcal M(\alpha))$ which recovers the Behrend-Fantechi one when $\mathcal M(\alpha)$ is an algebraic space, but exists even in the presence of semistable objects. 

We recall also that over the product $\underline{\mathcal M}(\alpha) \times X$, there exists a universal sheaf $\mathcal U$. The integrals we are interested in, are of the form \begin{equation}\label{desccorr1} \int\limits_{[\mathcal M]^{vir}}P(\pi_{1*}  (\operatorname{ch}_{i_k}(\mathcal U) \cup \pi_2^*\gamma_k)) \in \mathbb C\end{equation} where $\gamma_k$ are cohomology classes on $X$, and $P$ is what is known as a \textit{weight 0}-polynomial (see~\cite[Section 2.4]{BLM}). The latter condition is needed for the integrand to descend from $\underline{\mathcal M}(\alpha)$ to $\mathcal M$.

Our goal is, following here (a form of) Joyce's proposal from~\cite[Problem 7.71(e)]{Joy}, to give universal expressions for these integrals in terms of the analogous integrals associated to $\mathcal M(\beta)$ with $\operatorname{rank}(\beta) = 0$ and pure dimension $2$.

\subsection{} The main result of~\cite{FT} is that the virtual Euler characteristics of $\mathcal M(\alpha)$ for all $\alpha$ are uniquely determined by those of smaller rank, which inductively reduces all invariants to those of sheaves of rank $0$.

Our main theorem is its analogue, and is as follows.

\begin{thm}\label{mainthm} \textit{For any $\alpha \in K(X)$ with $\operatorname{rank}(\alpha) > 0$, any descendant integral given by the formula (\ref{desccorr1}) against Joyce's virtual class, can be written as a universal expression in the descendant integrals over various moduli spaces $\mathcal M(\beta)$ of Gieseker semistable sheaves of rank $0$ and pure dimension $2$.}
\end{thm}

\begin{remark} Modulo some technical restrictions on $X$, this is a counterpart of the main result of \cite{FT}. 
\end{remark}

\subsection{Outline.}
The main instrument of \cite{FT} is the theory of motivic wall-crossing for moduli spaces of perfect complexes on Calabi-Yau threefolds developed by Joyce and Song (\cite{JS}). Broadly speaking, their strategy is to pass from Gieseker stability to a certain Bridgeland-type weak stability condition, and perform the Joyce-Song manipulation for the moduli of objects in the corresponding heart of the derived category (which consists of complexes rather than of sheaves). This procedure is close to the replacement of a stable pair by its cokernel.

The natural counterpart of this method in the Fano world is the recent wall-crossing formalism of D. Joyce (\cite{Joy}). However, its methods work only for \textit{moduli spaces of sheaves} as opposed to complexes. Indeed, in \cite{Joy}, one needs a \textit{framing functor} to vector spaces, and, for the heart of a typical Bridgeland-type stability condition, it is unknown whether such a functor exists. Without this functor, it is unclear how to adapt the machinery of \cite{Joy}.

Fortunately, in the (even more) recent work \cite{KM} we reproved the $K$-theoretic counterpart of Joyce's results, by a different method which does not require any framing functor (this work was heavily inspired by previous investigations of Henry Liu, \cite{Liu}).

The main idea of this work is to apply \cite{KM} to the set-up of Feyzbakhsh-Thomas. In the process some new technical difficulties arise; in particular, handling the combinatorics of descendants (our solution heavily relies on the vertex algebra formalism of \cite{Joy}).

We hope that the present work can be of some interest as a hands-on introduction to \cite{KM}.

\subsection{The contents of the paper.} The outline of our work is as follows.

Section 2 contains preliminaries. We also set up the notation for various rigidifications and for $K$-theoretic decendents.

Section 3 consists of a summary of \cite{KM} and various works by D. Halpern-Leistner and his coauthors. We also check some assumptions like the existence of proper good moduli spaces for moduli spaces for Feyzbakhsh-Thomas stability conditions. The notable exception is the condition of quasi-smoothness which we establish in Section 5.

Section 4 contains a summary of \cite{FT} in just enough detail as needed for Section 6.

Section 6 is the main section, in which we do the actual ($K$-theoretic) wall-crossing.

Section 7 deduces the cohomological statement of our main theorem from the $K$-theoretic one, using various forms of Riemann-Roch fortified by the use of Adams operations.

The Appendix (joint with M. Moreira) proves an important but  technical statement about descendants using the vertex algebra formalism of \cite{Joy}. This part only occurs for Fano threefolds, and is not in the original Calabi-Yau case of \cite{FT}.

\subsection{Acknowledgments.} The present work owes everything to the unerring advice of Davesh Maulik, who rescued it from desperate situations at least twice, and to Miguel Moreira’s wise observation: “A map from the larger vector space to the smaller one is always surjective, what?”.

I am also very glad to thank Dominic Joyce and Richard Thomas for helping me to cross several walls, and for patiently listening to my attempts to explain the contents of the present paper, in June 2026.

I am grateful to Roman Bezrukavnikov, Chenjing Bu, Pavel Etingof, Andrés Ibáñez Núñez, Vasily Krylov, and Leonid Positselski for various related discussions.

\section{Preparation}\label{overview}

\subsection{Virtual structure sheaves and rigidifications}\label{virrig} Let us recall some features related to  derived algebraic stacks; these are virtual structure sheaves and rigidifications.

We will often work in the following situation: $\underline{\mathbfcal M}$ is a derived  stack of perfect complexes, sheaves, or pairs on the Fano variety $X$; $\iota^{cl}: \underline{\mathcal M} \to \underline{\mathbfcal M}$ is a classical truncation of $\underline{\mathbfcal M}$. $\underline{\mathbfcal M}$ is equipped with a natural free $B\mathbb G_m$-action.

In such a situation, as in~\cite[Definition 7.35]{Joy}, one forms a rigidification $\mathbfcal M$ of $\underline{\mathbfcal M}$ \footnote{Here our notation is different from the one in~\cite{HLD}!} by removing the $B\mathbb G_m$-action. The same can be done for $\underline{\mathcal M}$. There are natural morphisms $\pi: \underline{\mathcal M} \to \mathcal M$, and $\Pi: \underline{\mathbfcal M} \to \mathbfcal M$.

We suppose that $\underline{\mathbfcal M}$ is quasi-smooth, that is, its cotangent complex has the cohomological amplitude $[-1, 1].$ Then, $\mathbfcal M$ is quasi-smooth as well (see~\cite[Eq. (7.49)]{Joy}).

The essential properties for us are the following:

1) $\mathcal M$ is equipped with its \textit{$G$-theoretic virtual structure sheaf} $\mathcal O^{vir}_{\mathcal M} = \oplus H^i(\mathcal O_{\mathbfcal M})[i] \in G(\mathcal M)$ (see~\cite[Subsection 6.1]{HLH});

2) the classical truncation of $\mathbfcal M$ is $\mathcal M$;

3) each perfect complex $F$ on $\underline{\mathbfcal M}$ has a canonical weight decomposition $$F = \oplus F_i$$ so that $F_0$ descends uniquely to $\mathbfcal M$ (\cite[Subsection 1.0.4]{HLVer})

4) $R\Pi_*F = F_0$; 

5) For any $K$-theory class $F$ on $\mathcal M$, whenever their terms are defined, the equalities $$\chi(\underline{\mathbfcal M}, F) = \chi(\underline{\mathcal M}, \mathcal O^{vir}_{\mathcal M} \otimes \iota^{cl*}F) = \chi(\mathbfcal M, \Pi_* F) = \chi(\mathcal M, \pi_*(F) \otimes \mathcal O_{\mathcal M}^{vir})$$ hold. 

Moreover, these are all equal to
$$\chi(\underline{\mathbfcal M}, F_0).$$

These equalities always hold whenever $\mathcal M$ admits a proper good moduli space (see~\cite[Subsection 6.1]{HLH})

\subsection{$K$-theory and the descendants}\label{KThdescend}

Let us now introduce  $K$-theoretic descendants. 

The typical object for us will be some derived moduli stack $\underline{\mathbfcal M}$ of complexes of coherent sheaves on a smooth proper variety $X$, which admits a universal complex $\mathbfcal U \in \operatorname{Perf}(\underline{\mathbfcal M}\times X).$

For  a collection $\underline{\mathbfcal M}_1,$ $\ldots,$ $\underline{\mathbfcal M}_n$ of such derived stacks, each equipped with its own universal complex $\mathbfcal U_i$, let us consider (the pull-back of) $\mathbfcal U_i$ as a complex on $\prod \underline{\mathbfcal M}_i \times X$, and let us denote the projections from $\prod \underline{\mathbfcal M}_i \times X^n$ to $\prod \underline{\mathbfcal M}_i$, and to $X^n$ by $\pi_1$ and $\pi_2$, respectively.

As it will be explained in Section~\ref{cohtrans}, we would like to work with several copies of $X$. Namely, consider $\prod \underline{\mathbfcal M}_i \times \prod_{j=1}^k X.$ This carries universal complexes $\mathbfcal U_i^j$ where $i$ is as above, and $j$ indicates the copy of $X$ with respect to which the complex is taken. Let $\pi_1:\prod \underline{\mathbfcal M}_i \times \prod_{j=1}^k X \to \prod \underline{\mathbfcal M}_i$ be as above, and let $\pi_2^j$ be the projection to the $j$-th $X$-factor.

\begin{definition} \textit{The descendant classes} are  elements of $K(\prod_{i=1}^n \underline{\mathbfcal M}_i)$ which can be represented as linear combinations of the expressions of the form \begin{align}\label{descform}\mathcal F:= \pi_{1*} (\prod_{\lambda, i, j, k}\mathbb S_{\lambda} (\operatorname{Hom}_X(\mathbfcal U_i^k, \mathbfcal U_j^k))  \otimes \prod_j \pi_2^{j*}\upsilon_j \otimes \\ \prod_{\nu, a, b} \mathbb S_{\nu}( \mathbfcal U_a^b) \otimes \prod_{\kappa,c, d} \mathbb S_{\kappa}( \mathbfcal U_c^{d\vee})).\end{align}
where $\mathbb S_{\lambda}$ denotes the Schur functor, and all products are taken over arbitrary finite sets of partitions $\{\lambda,  \nu, \kappa\},$ indices $i,$ $j,$ $k$, integers $a, b, c, d$, and elements $\upsilon_j \in K(X).$
\end{definition}

Here, an expression of the form $\operatorname{Hom}(\mathbfcal U_i, \mathbfcal U_j)$ means the pull-back to $\prod \underline{\mathbfcal M}_i$ of the corresponding over $\underline{\mathbfcal M}_i \times \underline{\mathbfcal M}_j$ (see~\cite[7.36]{Joy}).

Further, by $K$-theory of some stack, e.g. by $K(\prod \underline{\mathbfcal M}_i)$, we mean in the present paper the direct sum of even and odd Blanc's topological K-theory groups of the category of perfect complexes on the stack in question (see \cite{Blanc} and \cite[Appendix A]{KM}).

By $\pi_{i*}$ we mean the derived  proper pushforward functor $R\pi_{i*}$ defined for these groups in \cite{K}. This agrees with the following convention. 

\noindent \textbf{Convention 1.} Throughout this paper, all $(-)^*$-, $(-)_*$-, $\operatorname{Hom}$-, $(-)^{\vee}$-, and $\otimes$-functors on the derived categories of perfect complexes on derived stacks and varieties, as well as the corresponding $K$-theoretic operations, are assumed to be the derived ones. In particular, $\pi_{1*}$ above is $R\pi_{1*}.$

\begin{remark} If the definition of descendants above is too abstract, one may think of $\iota^{cl*}(\mathcal F)$  instead of $\mathcal F$ where $\iota^{cl}$ is the classical truncation $\iota^{cl}: \prod \underline{\mathcal M}_i \to \prod \underline{\mathbfcal M}_i.$ As in~\cite[Definition 7.34]{Joy}, this is given by the same formula as (\ref{descform}) but now with  usual universal complexes $\mathcal U_i$ over $\underline{\mathcal M}_i$ instead of the derived universal complexes, that is, instead of $\mathbfcal U_i$ over $\underline{\mathbfcal M}_i$. \end{remark}

The numerical version of these constructions is as follows.

\begin{definition}\label{descnum} $K$-theoretic descendant integrals are the numbers given by formulas of the form \begin{equation}\label{descint}\chi(\prod \underline{\mathbfcal M}_i, \mathcal F) = \chi(\prod \underline{\mathcal M}_i, \iota^{cl*}(\mathcal F)\otimes \mathcal O^{vir})\end{equation}

where $\mathcal O^{vir} \in G(\prod \underline{\mathcal M}_i)$ is the virtual structure sheaf.
\end{definition}

\begin{remark}
In fact, we will later (that is, in Section~\ref{cohtrans}) see that the more reasonable idea is to take the integral not against $\mathcal O^{vir}$ but against the \textit{$\epsilon$-class} of \cite{KM}.
\end{remark}

When there is only one factor in the product of the moduli stacks which we consider, (\ref{descint})  should be compared with the cohomological descendants, that is, the expressions of the form 

\begin{align}
\int_{[\mathcal M]} P(\operatorname{ch}_{k_i}(\gamma_i)), \nonumber\\
P \in \mathbb C[\operatorname{ch}_{k_i}(\gamma) \ | \ \gamma \in H^*(X)] \nonumber\\
 \operatorname{ch}_{k_i}(\gamma) := \pi_{{\mathcal M}_*}(\operatorname{ch}_{k_i}(\mathcal U) \cup \pi_X^*(\gamma)) \label{cohdesc}
\end{align}

where $[\mathcal M]$ is some chosen virtual fundamental class on $\mathcal M$ (for example, the one introduced by Joyce in~\cite{Joy}), $\gamma \in H^*(X),$ and $P$ is such that $P(\operatorname{ch}_{k_i}(\gamma_i))$ descends from $\underline{\mathcal M}$ to $\mathcal M$. 

In this case, $P$ is known as a weight $0$-descendant, and the integral does not depend on the choice of the universal sheaf, see~\cite[Section 2.4]{BLM}.

\begin{remark}
In fact, in  cohomology, one may also consider several $\mathcal M$-factors, and write down an expression of the form (\ref{descform}). 

However, as it can be seen from~\cite{Gross} and \cite{BLM}, this does not give any new information. Indeed, the elements in question can be formally expressed as  polynomials  in terms of various classes of the form (\ref{cohdesc}) using the K\"{u}nneth theorem. 

The K\"{u}nneth theorem, however, does not hold in $K$-theory, so the consideration of several factors in formula (\ref{descform}) is no longer redundant.

We give a more detailed exposition in Section~\ref{cohtrans}.
\end{remark}

Let us end with the following convention.

\textbf{Convention 2.} Whenever we talk about a descendant (say, $\psi$) on some rigidified stack $\mathbfcal M$ (or a product of rigidified stacks), we really mean that $\psi$ is a descendant on $\underline{\mathbfcal M}$ which descends to $\mathbfcal M$.

\subsection{Some notation}\label{Kthsec}
Let us now introduce more of the main objects of our study. 

The general context is almost identical to the one of \cite{FT}, except now in the Fano setting.

As above, $X$ is a smooth projective, anticanonically polarized Fano $3$-fold over $\mathbb C$. We suppose that $X$ is in class $D$ in the sense of~\cite[Definition 4.3.6]{Gross}, that is, the natural morphism from semi-topological $K$-theory to the topological one is an isomorphism, $K^{s-t}(X) = K^{top}(X)$: for example, rational $3$-folds satisfy this condition.

Let $D^b(X)$ stand for the bounded derived category of $X$. Of particular interest to us will be the following subcategories of $D^b(X)$:

\begin{equation}\label{AbCat}\mathcal A_b  := \{E = [E_0 \xrightarrow{d} E_1] \in D^b(X) \ | \  \mu_H^+(\operatorname{ker}(d))\leq  b, \ \mu_H^-(\operatorname{coker}(d)) > b\}, \ b \in \mathbb R,\end{equation}

where $\mu_H$ is the notation for the slope of a coherent sheaf:
\begin{equation}
    \mu_H(\mathcal F) = 
    \begin{cases}
      \frac{\operatorname{ch}_1(\mathcal F).H^2}{\operatorname{ch}_0(\mathcal F).H^3} \ \text{if} \operatorname{ch}_0(\mathcal F) \neq 0, \ \\
      +\infty \ \text{if} \operatorname{ch}_0(\mathcal F) = 0, \ \\
    \end{cases}
\end{equation}
and the superscript $\pm$ is for maximal (resp., minimal) slope of the corresponding Harder-Narasimhan factor.

The following slope function (for any $w > \frac{1}{2}b^2$) defines a (weak) stability condition on $\mathcal A_b$:

\begin{equation}\label{slope}
    \nu_{b, w}(E) = 
    \begin{cases}
      \frac{\operatorname{ch}_2(E).H - \operatorname{ch}_0(E).wH^3}{\operatorname{ch}_1^{bH}(E).H^2} \ \text{if} \operatorname{ch}_1^{bH}(E).H^2 \neq 0, \ \\
      +\infty \ \text{if} \operatorname{ch}_1^{bH}(E).H^2 = 0, \ \\
    \end{cases}
\end{equation}

where $\operatorname{ch}_1^{bH}(E).H^2 = \operatorname{ch}_1(E).H^2 - \operatorname{ch}_0(E).bH^3$ is the first graded piece of the $H$-twisted Chern character $\operatorname{ch}^{bH}(E) = \operatorname{ch}(E)e^{-bH}.$

In \cite{BMT}, $\nu_{b, w}$ is proved to  have the Harder-Narasimhan property. Moreover, by \cite[Appendix C]{FT}, it coincides with a certain Bridgeland stability condition on the set of objects $E$ with $\nu_{b, w}(E) < +\infty$.

Throughout this note, we fix a triple of integers $p_1, p_2, q \in \mathbb N$, and a class $\mathbf{v} \in K(X)$ with the Chern character $\operatorname{ch}(\mathbf{v}) = (r, D, -\beta, -m) \in H^{2\bullet}(X, \mathbb Q)$.

In Section~\ref{stablewc} we will ask these data to satisfy the inequalities

$$D.H^2 = 0, -p_1 \leq \beta.H \leq p_2, m \leq q.$$

(We will get rid of the first assumption in Subsection~\ref{sec: arbcl}.)

We also pick $n = n(p_1, p_2, q, \mathbf v) \in \mathbb N$ large enough. More precisely, the bounds from \cite[p. 7]{FT}, as well as the bound from Subsection~\ref{rk2wc} below, should be satisfied. 

Our strategy (following \cite{FT}) is to relate the invariants of the class $\mathbf v$ with the ones associated with lower rank sheaves by carrying out wall-crossing for the class $v_n := \mathbf{v} - [\mathcal O(-n)]$.

\section{Wall-crossing}\label{wcp}
\subsection{The results of Karpov-Moreira}\label{KM}
The recent work ~\cite{KM} organizes the machinery of Subsection~\ref{HLRes} below as follows.

Suppose that all assumptions of \cite{KM} hold for some abelian category $\mathcal A$, for some subset $\mathfrak A$ of classes in $C(\mathcal A)$, and for some stability conditions $\mu, \mu_0$; in particular, all derived stacks $\mathbfcal M^{\mu}({\alpha}),$ and $\mathbfcal M^{\mu_0}({\alpha})$ of $\mu$- and $\mu_0$- semistable objects, respectively, for $\alpha \in \mathfrak A$, are quasi-smooth, and have proper good moduli spaces. \footnote{Note that we work with the rigidified stacks in the present subsection.} Suppose also that objects in classes from $\mathfrak A$ admits Harder-Narasimhan filtrations for both $\mu$, $\mu_0$. We will, finally, always suppose that \textit{any two} stability conditions we consider can be connected by \textit{a path of stability conditions} in the terminology of \cite{KM}.

There exists an  \footnote{Note that no stability condition is imposed here.} \textit{operational $K$-homology group} $$\mathbb K:=\bigoplus_{\alpha \in \mathfrak A} K_*(\mathbfcal M(\alpha)).$$ It should be thought of as an analogue of homology in the world of Blanc's $K$-theory (for the formal definition and more properties, see \cite{KM}). In particular, for each $\alpha$, there exists a natural evaluation functional $$K(\mathbfcal M(\alpha)) \otimes K_*(\mathbfcal M(\alpha)) \to \mathbb C.$$

The derived structure sheaves define canonical elements of $\mathbb K$, so it's natural to denote this functional by $$\tau \otimes \phi \mapsto \chi(\mathbfcal M(\alpha), \tau \otimes \phi):$$ indeed, when $\phi$ comes from $\mathcal O_{\mathbfcal M(\alpha)}$ the resulting functional is just $$\tau \mapsto \chi(\mathbfcal M(\alpha), \tau).$$

The following propositions are the main results of \cite{KM}.

\begin{proposition} For all classes $\alpha$ in $\mathfrak A$, and for both $\mu$ and $\mu_0$, there exist natural $\epsilon$-classes $\epsilon^{\mu}(\alpha)$, $\epsilon^{\mu_0}(\alpha) \in \mathbb K$ with the following properties.  \footnote{We reserve the right to abuse notation by removing either superscript $\mu$, or making $\alpha$ a subscript, or by doing both, whenever no confusion is likely to arise.}

\begin{enumerate}
\item[a)] If all $\mu$-semistable objects are, in fact, $\mu$-stable, the $\epsilon$-class is represented by the virtual structure sheaf:
$$\chi(\mathbfcal M(\alpha), \tau \otimes\epsilon^{\mu}(\alpha)) = \chi(\mathbfcal M^{\mu}(\alpha), \tau).$$

\item[b)] If the stacks $\mathbfcal M(\alpha)$ admit Joyce's virtual classes \cite{Joy} $[\mathcal M^{\mu}(\alpha)]^{vir}_{\mu} \in H_*(\mathbfcal M(\alpha))$ (and, in particular, the category $\mathcal A$ admits the so-called \textit{framing functor}; see~\cite{KM} for the discussion), the following formula holds

$$\chi(\mathbfcal M(\alpha), \tau \otimes \epsilon(\alpha)) = \int_{[\mathcal M(\alpha)]^{vir}_{\mu}} \operatorname{ch}(\tau)\operatorname{Td}(\mathbb T^{vir}_{\mathcal M(\alpha)}).$$
\end{enumerate}
\end{proposition}

\begin{definition} We will use the term '$K$-\textit{theoretic descendant integrals'} for the expressions of the form $$\chi(\prod \mathbfcal M(\alpha_i), \tau \otimes (\boxtimes \epsilon^{\mu_i}(\alpha_i)))$$ where $\tau$ is some descendant. It is shown in~\cite{KM} that this definition of descendant integrals is equivalent, via a natural combinatorial transformation, to the one given via derived structure sheaves in Subsection \ref{KThdescend}. We will also use the notation $\chi(\prod \mathbfcal M^{\mu_i}(\alpha_i), \tau \otimes \epsilon)$ for the same expression.
\end{definition}

 \begin{proposition} There exists a natural Lie bracket operation $[-,-]$ on $\mathbb K$ (graded with respect to $\mathfrak A$-grading) such that the following formula holds:

\begin{equation}\label{eq:epsilonwc}
\begin{aligned}
\varepsilon^{\mu_0}(\alpha)
    &=\sum_{\alpha_1+\ldots+\alpha_n=\alpha}\tilde U(\alpha_1, \ldots, \alpha_n; \mu, \mu_0) \\
    &\qquad\cdot [[\ldots [\varepsilon^{\mu}(\alpha_1), \varepsilon^\mu(\alpha_2)], \ldots, ],\varepsilon^{\mu}(\alpha_n)].
\end{aligned}
\end{equation}

    for some universal coefficients $\widetilde{U}(\alpha_1, \ldots, \alpha_n; \mu, \mu_0)$.
 \end{proposition}

 \begin{proposition} $[-,-]$ preserves descendants in the following sense: for any descendant $\tau$, there exists a natural descendant $\tau'$, so that $$\chi(\mathbfcal M(\alpha + \beta), \tau \otimes[\epsilon^{\mu}(\alpha), \epsilon^{\mu}(\beta)]) = \chi(\mathbfcal M(\alpha) \times \mathbfcal M(\beta), \tau'\otimes[\epsilon^{\mu}(\alpha) \boxtimes \epsilon^{\mu}(\beta)]),$$ and the same holds for iterated commutators.
 \end{proposition}

We will also need the following. Joyce, in \cite{Joy}, introduces a Lie bracket on $\bigoplus H_*(\mathbfcal M(\alpha))$ which is a natural analogue of our $K$-theoretic bracket. 

 \begin{proposition} If, again, Joyce's classes $[-]^{vir}$ are well-defined, $$\chi(\mathbfcal M(\alpha + \beta), \tau \otimes [\epsilon^{\mu}(\alpha),\epsilon^{\mu}(\beta)]) = \int_{[[\mathcal M(\alpha)]^{vir},  [\mathcal M(\beta)]^{vir}]}\operatorname{ch}(\tau)\operatorname{Td}(\mathbb T^{vir}_{\mathcal M(\alpha+\beta)}),$$

 and the natural analogue also holds for iterated commutators.
 \end{proposition}

\begin{remark} We discuss further the relation between homological and $K$-theoretic invariants in Subsection~\ref{cohtrans}.
\end{remark}

Let us finally mention that one can take  exterior products of elements of $\mathbb K$. Hence, the following minor generalization of the above results holds

\begin{proposition}\label{prods} Let $\mathbfcal M^{\mu_i}(\beta_i)$, $i = 1, \ldots, n$ be a collection of moduli stacks of objects in  classes $\beta_i$ for some $\beta_i \in \mathfrak A$ and stabilities $\mu_i$. Suppose also that $\mu_0$ is one more stability condition. Then the following equality holds:
\begin{multline}
\epsilon_{\beta_1}^{\mu_1} \boxtimes \ldots \boxtimes \epsilon_{\beta_i}^{\mu_0} \boxtimes \ldots \boxtimes \epsilon^{\mu_n}_{\beta_n} \\
= \epsilon_{\beta_1}^{\mu_1} \boxtimes \ldots \boxtimes
\left(\sum_{\alpha_1+\ldots+\alpha_l=\beta_i}\tilde U(\alpha_1, \ldots, \alpha_l; \mu_i, \mu_0)
\cdot [[\ldots [\varepsilon_{\alpha_1}^{\mu_i}, \varepsilon_{\alpha_2}^{\mu_i}], \ldots, ],\varepsilon_{\alpha_l}^{\mu_i}]\right)
\boxtimes \ldots \boxtimes \epsilon_{\beta_n}^{\mu_n}.
\end{multline}

for some universal coefficients $\tilde U$.

\end{proposition}

We can now formulate the $K$-theoretic analogue of  Theorem \ref{mainthm}. 

\begin{thm}\label{Kthmain}
Let $\mathbfcal M(\alpha)$ be a rigidified stack of sheaves in a given class $\alpha$ as above. Let $\epsilon_{\alpha}$ be the $\epsilon$-class corresponding to Gieseker stability. Then, for every descendant $\tau$ on $\mathbfcal M(\alpha)$, there exists a universal expression of the form

$$\chi(\mathbfcal M(\alpha),\tau \otimes  \epsilon_{\alpha}) = \sum_i V_{\alpha_{i,j}} \chi(\prod_j \mathbfcal M(\alpha_{i, j}), \psi_i \otimes (\boxtimes \epsilon_{\alpha_{i,j}})) $$ 

where $\psi_i$ are some descendants, $V_{\alpha_{i ,j}}$ are universal coefficients, $\alpha_{i, j}$ are of rank $0$, pure dimension $2$, and all $\epsilon$-classes on the RHS are taken in the sense of Gieseker stability as well.
\end{thm}

Our strategy is to prove it first, and then to deduce the cohomological statement as in Section \ref{cohtrans}.

\subsection{Pairs}\label{KMpairs}
  Let $\mathcal A$ be an abelian category, and let $\alpha \in \mathfrak A$ be a topological type as above.

  Here, for simplicity, we suppose that $\mathcal A$ is the category of coherent sheaves on $X$, and $\mu$ is either Gieseker stability, or tilt-stability (see Subsection \ref{sec: stabcond}).

   Recall that Joyce (see~\cite[Section 8.1]{Joy}) considers an abelian category $\acute{\mathcal A}$ of pairs \begin{equation}\label{pair}\rho: V \otimes_{\mathbb C} L \to E\end{equation} where $E$ is a coherent sheaf on $X$.

Here, $V$ is a finite-dimensional vector space, and $L$ is a fixed line bundle.

By the ``topological type'' of an object (\ref{pair}) from $\acute{\mathcal A}$, we will mean the pair $(\alpha, d)$ where $\alpha \in K(X)$ is the $K$-theory class of $E$, and $d$ is $\operatorname{dim}(V)$.

Joyce introduces the following weak stability condition $\mu^s$ on $\acute{\mathcal A}$:

\[
\mu^s(\alpha, d) =
\begin{cases}
(\mu(\alpha), d/\operatorname{rk} \alpha), & \operatorname{rk} \alpha \neq 0, \\
(\infty, 1), & \operatorname{rk} \alpha = 0.
\end{cases}
\]
(The two factors are ordered lexicographically, and $\mu$ stands for $\mu$-slope.)

Let $\underline{\mathbfcal P}_n(\alpha, d)$ denote the (derived) stack of $\mu^s$-semistable pairs of  class $(\alpha, d)$ with $L = \mathcal O(-n)$. We write $\underline{\mathbfcal P}_n(\alpha)$ for $d=1.$ The characteristic feature of this stack is that for $\mu^{s}$, all semistable objects are stable; in particular, in the sense explained above, $\epsilon_{\mathbfcal P_n(\alpha)} = \mathcal O_{\mathbfcal P_n(\alpha)}.$

In~\cite{KM}, as a byproduct of the aforementioned general results, we also obtain the following: 

\begin{proposition}\label{pairs} For any descendant $\tau$ and for sufficiently large $n$, there exists an expression of the form
$$\chi(\mathbfcal M(\alpha), \tau \otimes \epsilon(\alpha)) = \chi(\mathbfcal P_n(\alpha), \tau' \otimes \epsilon_{\mathbfcal P_n(\alpha)})  + \mathcal U$$
where $\mathcal U$ is a universal expression in terms of descendant integrals over the products of moduli spaces in $\mathcal A$, each factor of which has strictly lower rank than $\alpha$. The symbol $\tau'$ here is some $K$-theoretic descendant which can be formally determined by $\tau.$
\end{proposition}

\subsection{The results of Halpern-Leistner}\label{HLRes}

This subsection is devoted to results of Halpern-Leistner and his coauthors which underlie the machinery of \cite{KM}. 

We include it to make the exposition more self-cobtainedd. We list some precise references below; a general reference is ~\cite{HLstruc}.

\subsubsection{$\Theta$-stratifications}\label{Thetastr}

Let $\mathcal M$ be an algebraic stack over $\mathbb C$ which is locally of finite type, quasi-separated, and has separated inertia with affine relative automorphism  groups. We will mostly be interested in the case where $\mathcal M$ is (a non-rigidified) moduli stack of objects in some abelian category.

Now we consider $\Theta:=\mathbb A^1/\mathbb G_m$, and the stacks:

i) $\operatorname{Grad}(\mathcal M) := \operatorname{Maps}(B\mathbb G_m, \mathcal M)$ which can be interpreted as \textit{the stack of graded points of $\mathcal M$}: this should be thought of as the stack of graded objects in the same category;

ii) $\operatorname{Filt}(\mathcal M) := \operatorname{Maps}(\Theta, \mathcal M)$, that is, \textit{the stack of filtered points of $\mathcal M$}.

There is a canonical morphism of taking the associated graded, obtained by restriction along the inclusion $0/\mathbb G_m \to \Theta$, $$\operatorname{gr}: \operatorname{Filt}(\mathcal M) \to \operatorname{Grad}(\mathcal M).$$

Also, we will use the evaluation map $\operatorname{ev}_1: \operatorname{Filt}(\mathcal M) \to \mathcal M$ given by restricting the map to the open substack $\operatorname{pt} = (\mathbb A^1 \setminus \{0\})/\mathbb G_m$ inside $\Theta.$

\begin{definition} A \textit{$\Theta$-stratum} in $\mathcal M$ is a union of connected components $\mathfrak S \subset \operatorname{Filt}(\mathcal M)$ such that the morphism $\operatorname{ev}_1: \mathfrak S \to \mathcal M$ is a locally closed embedding.
\end{definition}

\begin{definition} A $\Theta$-stratification of $\mathcal M$ consists of 

1) a well-ordered set $\Gamma$ and a collection of open substacks $\mathcal M_{\leq \alpha}$ for $\alpha \in \Gamma:$ $\mathcal M_{\leq \alpha} \subseteq \mathcal M_{\leq \alpha^{'}}$ for $\alpha < \alpha^{'}$, and $\mathcal M = \bigcup \mathcal M_{\leq \alpha}$;

2) a $\Theta$-stratum in each $\mathcal M_{\leq \alpha}$, $\operatorname{ev}_1: \mathfrak S_{\alpha } \to \mathcal M_{\leq \alpha}$, such that $$\mathcal M_{\leq \alpha} \setminus \operatorname{ev}_1(\mathfrak S_\alpha) = \mathcal M_{<\alpha}:=\bigcup \limits_{\alpha^{'} < \alpha} \mathcal M_{\leq \alpha^{'}};$$

3) for every point $x \in |\mathcal M|$, the set $\{\alpha \in \Gamma \ | \ x \in \mathcal M_{\leq \alpha}\}$ is finite.
\end{definition}

We will also need a weaker notion of pseudo-$\Theta$-stratification, see Remark~\ref{pseudo}, which is defined in~\cite[Definition 3.34]{Joy}.

\begin{definition}
    A pseudo-$\Theta$-stratification consists of the same amount of data as a $\Theta$-stratification. However, now

    i) the set $\Gamma$ is only partially ordered;
    
    ii) though $\mathfrak S_{\alpha}$ is assumed to exist, no distinguished choice is made.
\end{definition}

\begin{remark}\label{pseudo} The important class of $\Theta$-strata for moduli stacks comes from substacks of objects having a prescribed Harder-Narasimhan type (with the filtered structure encoding the Harder-Narasimhan filtration). It is known that, for example, this defines a pseudo-$\Theta$-stratification when Harder-Narasimhan stratifications are taken for Gieseker stability, see Subsection~\ref{CaseGies} below for the references.

However, as far as we understand, it is still not known whether, in this particular example, the pseudo-$\Theta$-stratification can be canonically upgraded to $\Theta$-stratification. See ~\cite{GHR} for some partial results in this direction.
\end{remark}

We also assume that $\Gamma$ contains a minimal element $0$. We will refer to $\mathcal M_{\leq 0}$ as  the \textit{semistable locus} $\mathcal M^{ss}$, and to its complement as the \textit{unstable locus} $\mathcal M^{us}$.

Halpern-Leistner introduces a notion of \textit{the center of a stratum.}

\begin{definition} Let $\mathcal S$ be a (weak) $\Theta$-stratum (or pseudo-$\Theta$-stratum) corresponding to $\alpha \in \Gamma$. Let $\sigma$ be the natural morphism $$\sigma:\operatorname{Grad}(\mathcal M) \to \operatorname{Filt}(\mathcal M)$$ sending a graded object to the same object equipped with the canonical filtration associated to its grading.
Then, one defines the center of $\mathcal S$ as $\sigma^{-1}(\mathcal S)$. We will denote it by $\mathcal Z(\alpha)^{ss}$.
\end{definition}

Let $\mathbfcal M$ be a derived algebraic stack. 

Then, the mapping stack $\operatorname{Filt}(\mathbfcal M)$ inherits a derived structure, and, moreover, specifying a derived $\Theta$-stratum, whose definition is as above, is equivalent to specifying one for the underlying classical stack. The same holds for pseudo-$\Theta$-stratifications, and all other objects discussed above.

See the proof of ~\cite[Lemma 1.2.3]{HL} for more details.

\subsubsection{Non-Abelian Localization}\label{nonabloc}

Let $\mathbfcal M$ be a derived algebraic stack, which 

i) is of finite type, with affine diagonal over $\mathbb C$;

ii) has a pseudo-$\Theta$ stratification so that centers of strata have proper good moduli spaces;

iii) is quasi-smooth;

iv) has a proper good moduli space. 

Then, the expressions of the form $$\chi(\mathbfcal M, F), \ \chi(\mathbfcal M^{ss}, F), \ \chi(\mathbfcal Z(\alpha)^{ss}, F)$$ are defined for $F \in$ $\operatorname{Perf}(\mathbfcal M)$, $\operatorname{Perf}(\mathbfcal M^{ss})$, $\operatorname{Perf}(\mathbfcal Z(\alpha)^{ss})$, respectively.

Let us consider some stratum $\alpha$. Its center is a substack of the mapping stack
\[
\operatorname{Maps}(B\mathbb G_m, \mathcal M).
\]
In particular, as in~\cite[Section 5]{HLrem}, the pullback
\[
\mathbb L_{\mathbfcal M}|_{\mathbfcal Z(\alpha)^{ss}}
\]
has a natural weight decomposition with integer weights, the so-called weight-grading.

Let $\mathbb L(\alpha)^+$ and $\mathbb L(\alpha)^-$ stand for the positive and negative parts of this decomposition respectively. The following statement is known as the Non-Abelian Localization theorem of Halpern-Leistner (see~\cite{HLrem}).

\begin{theorem}\label{nonabl} Let $\mathbfcal Z(\alpha)^{ss}$ be the centers of strata as above. For each $\alpha$, let $E(\alpha)$ be the complex

\begin{equation}\label{complcorr}E(\alpha):=\operatorname{Sym}(\mathbb L(\alpha)^- \oplus (\mathbb L(\alpha)^+)^{\vee}) \otimes \operatorname{det}(\mathbb L(\alpha)^+)^{\vee}[-\operatorname{rk} \mathbb L(\alpha)^+].\end{equation}

Then, for any $F \in \operatorname{Perf}(\mathbfcal M)$ the formula

\begin{equation}\label{nonablform}\chi(\mathbfcal M, F) = \chi(\mathbfcal M^{ss}, F|_{\mathbfcal M^{ss}}) + \sum  \chi(\mathbfcal Z(\alpha)^{ss}, F|_{\mathbfcal Z(\alpha)^{ss}} \otimes E(\alpha))\end{equation}

holds.
\end{theorem}

\begin{remark}\label{rem: weights} In principle, the formula~\ref{complcorr} does not define a complex with coherent cohomology sheaves. However, we work on $\operatorname{Perf}(\mathbfcal Z(\alpha)^{ss})$, and this category acquires the weight-grading.

It is known (see Subsection~\ref{virrig}) that only weight $0$ complexes may have non-trivial Euler characteristics.

The weights of $E(\alpha)$ are bounded from above. In particular, for any $F$ with a bounded set of weights, the formula~\ref{nonablform} makes sense since the weight-pieces have coherent cohomology sheaves.
\end{remark}

\begin{remark} The reference for the proof is~\cite{HLnew}. The case of \textit{pseudo}-$\Theta$-stratified stacks is addressed in~\cite{KM}.
\end{remark}

\begin{remark} In the cases of interest to us, the centers of the strata will be products of lower rank moduli spaces. In these situations, $ \mathbb L(\alpha)^+$ and $\mathbb L(\alpha)^-$ will have a natural description in terms of the cotangent complexes of factors (see Proposition~\ref{domstabcon}).
\end{remark}

\subsubsection{The case of $\nu_{b, w}$}\label{BridStab}

The goal of the rest of this section is to explain what $\Theta$-stratifications look like on the moduli spaces of objects in the categories from Subsection~\ref{Kthsec},  and to  check the assumptions i), ii), iii), iv) from Subsection~\ref{nonabloc}.

Let $\mathcal A$ be the abelian category $\mathcal A_b$.

We choose a class $\alpha$ in $C(\mathcal A)$  with $\nu_{b, w}(\alpha) \neq 0$, and two weak stability conditions  $\tau = (b, w)$ and $\tau_0 = (b, w_0)$ so that $\tau_0$ \textit{weakly dominates} $\tau$  \textit{for class $\alpha$} in the following sense (see \cite[p.30]{FT}). \footnote{In particular, this holds when $\alpha$ is a torsion-free class in the sense of~\cite{HLstruc}, $\tau_0$ is a (weak) stability on some wall, and $\tau$ is a (weak) stability condition off this wall but sufficiently close to it.}

\begin{definition}\label{weakdom}
We will say that $\tau_0$ \textit{weakly dominates} $\tau$  \textit{for class $\alpha$} iff
$$\tau(\gamma) \leq \tau(\beta) \Rightarrow \tau_0(\gamma) \leq \tau_0(\beta)$$ when
each of the classes $\gamma$ and $\beta$ is either $\alpha$, or a class of a $\tau$-semistable HN-factor of some $\tau_0$-semistable objects of class $\alpha$ (compare with ~\cite[Definition 3.8]{Joy}).
\end{definition}

 In particular, in this case ${\underline{\mathcal M}}_{\tau}(\alpha)^{ss} \subseteq \underline{\mathcal M}_{\tau_0}(\alpha)^{ss}.$

Let us, then, consider the stacks ${\underline{\mathcal M}}_{\tau}^{ss}(\alpha)$  of the $\tau$-semistable complexes in the category $\mathcal A_b$ with some fixed topological type $\alpha \in K(X)$. Similarly, consider ${\underline{\mathcal M}}_{\tau_0}^{ss}(\alpha)$.

We assume that $\alpha$ is a torsion-free class in the sense of~\cite[Definition 6.4.9]{HLstruc}, that is, $\nu_{b, w}(\alpha) < +\infty$. Moreover, we assume that $\alpha$ is positive: either $\operatorname{ch}_0(\alpha) > 0,$ or $\operatorname{ch}_0(\alpha) = 0,$ and $\operatorname{ch}_1(\alpha).H^2 > 0.$

We now verify the assumptions i), ii),  iii) and iv) from Subsection~\ref{nonabloc} for ${\underline{\mathcal M}}_{\tau}(\alpha)$, ${\underline{\mathcal M}}_{\tau_0}(\alpha)$, and also check that these stacks have proper good moduli spaces. 

\textit{Assumption i).} The finiteness for this assumption follows from~\cite[Theorem C.5]{FT}. The diagonal is affine by ~\cite[Proposition 6.2.7]{HLstruc}. This proposition, however, requires the property of \textit{generic flatness}; this property was checked in ~\cite[Proposition C.4]{FT}.

\textit{Assumption ii).} For the assumption ii), we claim   that ${\underline{\mathcal M}}_{\tau_0}(\alpha)^{ss}$ has a $\Theta$-stratification corresponding to the Harder-Narasimhan filtrations of $\tau_0$-semistable objects by the $\tau$-semistable ones.

The strata correspond to the filtered objects which come as  Harder-Narasimhan filtrations. However, the order on the set of these strata, and the \textit{weights} \footnote{Note that the data of the map $\Theta \to {\underline{\mathcal M}}$ remembers not only the filtration but also the $\mathbb Z$-grading on the associated grading; the degrees in question are known as the corresponding \textit{weights}.} associated to the terms of the filtration, are much more sophisticated. The existence of the compatible choices so that all of the axioms are satisfied follows from the results of~\cite[Theorem 6.5.3]{HLstruc} (see also~\cite[Theorem 7.27, Theorem 7.29]{AHLH}). 

\begin{remark} In fact, one should be careful here: the references above state Theorems for  Bridgeland stability conditions, and $\nu_{b, w}$ is not a Bridgeland stability condition.

However, it is still given by a central charge. This central charge is not in \cite{FT}; however, we discuss it in the formula (\ref{centcharge}) below.

In particular, the argument of \cite[Theorem 6.5.3]{HLstruc}, ~\cite[Theorem 7.27, Theorem 7.29]{AHLH}, applies. To see this,  one has to check three properties to use it: a) algebraicity, that is, the fact that every stability condition is equivalent to one with $Z_{b, w}(\alpha) \in \mathbb Q + i\mathbb Q$, b) generic flatness, c) the boundedness of quotients.

For a), we may work under the desired assumption of $Z_{b, w}(\alpha) \in \mathbb Q + i\mathbb Q$ since the walls are defined over $\mathbb Q$ by the explicit form (\ref{centcharge}) of the central charge.

b) was proven in~\cite[Proposition C.4]{FT}.

c) follows from~\cite[Lemma 6.5.5]{HLstruc} and ~\cite[Theorem C.5]{FT}.
\end{remark}

As in~\cite[Lemma 1.2.3]{HL}, the existence of the derived $\Theta$-stratification is now immediate from the existence of the classical one and of the derived structure on ${\underline{\mathcal M}}_{\tau_0}^{ss}(\alpha)$ (see~\cite{TV}). (The corresponding derived stacks are defined as open substacks of the natural derived enhancement of the Toën-Vaquié stack for $D^b(X)$, see~\cite{TV}.)

We are done with the condition ii) from Subsection~\ref{nonabloc}: the existence of proper good moduli spaces for centers of strata follows from the rest of the present section and Proposition \ref{centergp}.

\textit{Assumption iii).} The condition iii), that is, the quasi-smoothness of ${\underline{\mathcal M}}_{\tau}(\alpha)^{ss}$ and ${\underline{\mathcal M}}_{\tau_0}(\alpha)^{ss}$, will be proved for positive classes (see Theorem~\ref{Extv}) in Subsection~\ref{MozgSec}: it follows from Theorem~\ref{Extv} as in~\cite[Definition 7.45]{Joy} since $h^2(\iota^{cl*}\mathbb L_{{\underline{\mathcal M}}}^{\vee})|_{[E]} = \operatorname{Ext}^3(E, E)=0.$

\textit{Assumption iv).} Finally, let us see that the stacks of the form ${\underline{\mathcal M}}_{\tau}^{ss}(\alpha)$ for $\tau \in U$ admit proper good moduli spaces. This is essentially contained in the work of Alper-Halpern-Leistner-Heinloth but we will go along the lines of similar arguments from ~\cite{Joy} instead.

First of all, $\tau$ is additive in the sense of~\cite[Subsection 3.3.5]{Joy} (see~\cite[7.27]{AHLH}). Moreover, ${\underline{\mathcal M}}_{\tau}^{ss}(\alpha)$ are bounded by~\cite[Theorem C.5]{FT}, and, hence, all of the strata are bounded as well by Proposition~\ref{centergp} below, and by ~\cite[Lemma 3.16]{ToK3}.

By ~\cite[Theorem 3.43]{Joy}, it follows that, to check the existence  of the proper good moduli space for ${\underline{\mathcal M}}_{\tau}(\alpha)$, one has only to show that the category $\mathcal A_b$ is of compact type.

To see this, we realize $\mathcal A_b$ as a heart $D^b(X)^{\heart}$ of some $t$-structure on $D^b(X)$ (\cite[Lemma 6.1]{Bridge}, see also ~\cite[p. 1]{FT}). Now,~\cite[Section 6.2]{HLstruc} shows that $(D(X)_{qc})^{\heart} = \operatorname{Ind}(\mathcal A_b)$ \footnote{This equality holds since, by ~\cite[Section 6.2]{HLstruc}, this category is compactly generated and has all filtered colimits.} is locally Noetherian and has $\mathcal A_b$ as the subcategory of its compact objects, and we are done.

\subsubsection{The case of the Gieseker stability}\label{CaseGies}

We will also apply the nonabelian localization theorem to the pseudo-$\Theta$-stratification induced by Gieseker stability.

Fix $\alpha \in K(X)$ of $\operatorname{rank} > 0$ or of $\operatorname{rank} = 0$ and pure dimension $2$.

 Suppose  that $\tau_{Gies}$ is a Gieseker stability condition, and $\tau$ is a stability condition of the form $\nu_{b, w}$ so that all $\tau$-semistable objects of class $\alpha$ are sheaves, and $\tau$ weakly dominates $\tau_{Gies}$ in the sense of Definition \ref{weakdom}.

First of all, the assumption i) holds in this situation by \cite[Proposition 7.23, Example 3.36]{Joy}. 

The assumption ii) takes the following form: it says that there exists a pseudo-$\Theta$ stratification of
\[
{\underline{\mathcal M}}_{\tau}(\alpha)^{ss}
\]
corresponding to the Harder-Narasimhan filtrations of $\nu_{b, w}$-semistable sheaves by Gieseker semistable sheaves, similarly to the situation from Subsection~\ref{BridStab}. This follows from~\cite[Example 3.36]{Joy}.

The assumption iii) was already checked.

The assumption iv) follows from~\cite[Subsection 7.3.4]{Joy}.

\subsubsection{The centers}\label{Sec: centers} 
We claim that when $\tau$, $\tau_0$ are sufficiently close $(b, w)$-conditions so that $\tau_0$ weakly dominates $\tau$ for some torsion-free class $\alpha$ (in the sense of Subsection~\ref{BridStab}) the centers of Harder-Narasimhan strata of ${\underline{\mathcal M}}_{\tau_0}^{ss}(\alpha)$ with respect to $\tau$ have a natural description in terms of ${\underline{\mathcal M}}_{\tau}^{ss}(\beta)$ for various classes $\beta$.

Indeed, the following abstract statement was proved in~\cite{KM}.

\begin{proposition}\label{domstabcon}
Suppose that $\mathcal A$ is an abelian category with two weak stability conditions, say, $\tau_0$ and $\tau$. Suppose that, for some class $\alpha$, $\tau_0$ weakly dominates $\tau$ in the sense of Subsection~\ref{BridStab}.

Suppose also that $\mathcal A$ satisfies the conditions of \cite[Section 2]{KM}. In particular, $\tau_0$-semistable objects in some class $\alpha$ admit $\tau$-Harder-Narasimhan filtrations, and this gives a $\Theta$-stratification of ${\underline{\mathcal M}}_{\tau}^{ss}(\alpha)$ (for some choice of the order, and some choice of the weights, as in Subsection~\ref{BridStab}). Then, the following holds.

 Consider the $\Theta$-stratification of ${\underline{\mathcal M}}^{ss}_{\tau_0}(\alpha)$ induced by Harder-Narasimhan decompositions for the stability condition $\tau$. Then, the centers of strata are of the form
\begin{equation}\label{desccorr} {\underline{\mathcal M}}^{ss}_{\tau}(\beta_1) \times \ldots \times {\underline{\mathcal M}}^{ss}_{\tau}(\beta_n),\end{equation}
and for each set $\{\beta_i\}$, $\sum \beta_i = \alpha$, and $\tau_0(\beta_i) = \tau_0(\beta_j)$ for all $i, \ j$.  The corresponding equality holds for the natural derived enhancements in the sense of~\cite{TV} as well (whenever they exist).
\end{proposition}

The  corollary is

\begin{proposition}\label{centergp} The centers of the strata from Subsection \ref{BridStab} admit proper good moduli spaces.
\end{proposition}

Let us return to the Non-Abelian Localization Theorem. As we have seen in Subsection~\ref{BridStab}, it applies to the $\Theta$-stratification induced by Harder-Narasimhan structure with respect to $\tau$ on ${\underline{\mathcal M}}^{ss}_{\tau_0}(\alpha)$. But one has to decipher the definition of $E(\alpha)$.

First of all, as we noted above, the cotangent complex $\mathbb L_{\underline{\mathbfcal M}_{\tau_0}(\alpha)}$ has the fibers

$${\mathbb L_{\underline{\mathbfcal M}}}_{\tau_0}(\alpha)|_{[E]} = \operatorname{RHom}_X(E, E)^{\vee}[-1].$$

Now, suppose that $E \in \mathbfcal Z(\alpha)^{ss},$ $E = \bigoplus E_i,$ $E_i \in \underline{\mathbfcal M}_{\tau}(\beta_i)^{ss}$ as above. With respect to the weight grading, each $E_i$ sits in its own degree $w_i$ , $w_1 < w_2 < \ldots < w_n$ (the Harder-Narasimhan condition). Now, 

$$\operatorname{RHom}_X(E, E)^{\vee}[-1] = \bigoplus \operatorname{RHom}_X(E_i, E_j)^{\vee}[-1],$$ and the summand $\operatorname{RHom}_X(E_i, E_j)^{\vee}$ has the weight $w_i - w_j$.

In particular, we get the following formula for $E(\alpha)$:

\begin{align}\label{Ea} E(\alpha) =  \operatorname{Sym}^{\bullet}\left(\bigoplus\limits_{i < j} \operatorname{RHom}(E_i, E_j)^{\vee}[-1] \oplus \operatorname{RHom}(E_j, E_i)[1]\right) \otimes \mathcal L_E[\eta_E]\end{align}
where $\mathcal L_E$ is the line bundle $\operatorname{det}(\bigoplus\limits_{i < j} \operatorname{RHom}(E_j, E_i)^{\vee}[-1])^{\vee}$, and $\eta_E$ is the number $$-\operatorname{rank} (\bigoplus\limits_{i < j} \operatorname{RHom}(E_j, E_i)^{\vee}[-1])^{\vee}.$$

\textit{This formula is the main ingredient of the results recalled in Section \ref{KM}.}

\section{The results of Feyzbakhsh-Thomas}
\subsection{The summary of~\cite{FT}.}\label{summ} We give here a brief overview of some aspects of the proof from \cite{FT}. As we already mentioned, our strategy is similar. 

The authors of \cite{FT} work under different assumptions on $X$, assuming that it is Calabi-Yau.  However, their goal is the same, that is, it is to express the enumerative invariant $J(\mathcal M(\alpha))$, which is a Behrend-weighted Euler characteristic, using the lower rank moduli spaces. This gives the $\operatorname{rank} r \to \operatorname{rank} 0$-type result of the same flavour as  above.

One of the main technical ingredients of their proof is  Joyce's motivic Hall algebra (see~\cite{Jo2}). The use of this technique is as follows. For a wall $l$ in the space $\mathbb R^2$ of the stability parameters, and for  parameters $(b, w_+)$ and $(b, w_-)$ just above and below $l$, it makes it possible to express the invariant $J_{b, w_+}(\alpha)$ through $J_{b, w_-}(\alpha)$ and $J_{b, w_-}(\alpha_i)$. Here $\alpha_i$ are the topological types of the destabilizing objects on the wall. The coefficients occurring in the resulting Joyce's formula are universal.

More precisely, the authors of~\cite{FT} proceed as follows. First of all, it turns out that the semistable and destabilizing objects they work with are still \textit{mostly} sheaves, though this statement requires a non-trivial analysis of \cite{FT}. 

In fact, essentially, complexes occur in the analysis of \cite{FT} only at one particular wall, the so-called Joyce-Song wall $l_{JS}$: here objects of the form $F \oplus \mathcal L[-1]$ may arise for a \textit{sheaf} $F$ and a line bundle $\mathcal L$.

In particular, outside $l_{JS}$, destabilizing subobjects usually have the rank strictly lower than the ones that they destabilize. The Hall algebra machinery, together with induction on rank, helps to reduce the $\operatorname{rank} r \to \operatorname{rank} < r$ problem to the case of $J_{b, \infty}(\alpha)$. This stability is known as \textit{large-volume stability} and turns out to be equivalent to tilt-stability, which can then be wall-crossed to Gieseker stability, and hence to \textit{any sufficiently general} $J_{b, w}(\alpha)$.

Thus, it suffices to express Behrend Euler characteristics $J_{b, w}(\alpha)$ with $\operatorname{rk} \alpha = r \in \mathbb N$, in terms of similar invariants calculated for \textit{some} $(b, w)$ via the similar invariants of lower rank. For this, one chooses $(b, w_+)$ right above $l_{JS}$ and $(b, w_-)$ right below $l_{JS}$, and does the $l_{JS}$-wall-crossing $J_{b, w_+}(\beta) \to J_{b, w_-}(\beta)$ for a suitably defined class $\beta$ of rank $r-1$.

It turns out that the resulting wall-crossing formula for $J_{b, w_+}(\beta) - J_{b, w_-}(\beta)$ contains only one summand which is not an expression in terms of invariants corresponding to $\operatorname{rank} < r$ classes. This summand is a non-zero multiple of $J_{b, \infty}(\alpha)$ (which corresponds to the aforementioned objects with factors $F$ and $\mathcal L[-1]$).

Hence, the described procedure gives a way to express $J_{b, \infty}(\alpha)$ through the lower rank enumerative invariants. 

In the rest of the present section, we overview some technical details of the procedure which will be useful for us.

\subsection{Technical points}\label{FTsec}

\subsubsection{Semistable objects are sheaves}\label{sec: FT1}

In \cite{FT}, Feyzbakhsh and Thomas prove the following description of the semistable/destabilizing objects inside $\mathcal A_b$.

 Let $\mathbf v$ be a class in $K(X)$ of some rank $r$, and let $v_n$ be $v - [\mathcal O(-n)]$ for $n$ large enough, so that $(v, n)$ satisfies equalities and inequalities from \cite[p. 7]{FT}.

 The following is one of the main results of \cite{FT}.

\textbf{Claim.} Let $U$ be the subset of $\mathbb R^2$ consisting of points $(b, w)$ with  $w > \frac{1}{2}b^2$. Given $(b, w) \in U$ with $b < \frac{n}{r-1}$, any $\nu_{b, w}$-semistable $E \in \mathcal A_b$ of class $v_n$ is either a sheaf, or of the form $F \oplus \mathcal O(-n)[1]$. Here, $F$ is a sheaf of class $\mathbf{v}$ which is semistable in the large volume chamber.
The latter  possibility happens only on the wall $l_{JS} \subseteq \mathbb R^2$ known as \textit{the Joyce-Song wall}. This is the line joining $\Pi(v_n)$ and $\Pi(\mathcal O(-n)[1])$ for $\Pi$ being the projection $$K_H(X) \setminus \{E \ | \ \operatorname{ch}_0(E) = 0\} \to \mathbb R^2,$$ $$E \mapsto (\frac{\operatorname{ch}_1(E).H^2}{\operatorname{ch}_0(E).H^3}, \frac{\operatorname{ch}_2(E).H}{\operatorname{ch}_0(E).H^3}).$$ (This line will play the major role in Section~\ref{stablewc}.)

\subsubsection{Destabilizing objects are sheaves}\label{JS}

\begin{definition}\label{poscl}
As in \cite{FT}, we will call a class $v \in K_H(X)$ \textit{positive} if, for $\operatorname{ch}(v) = (v_0, v_1, v_2, v_3)$, one has either:

a) $v_0 > 0,$ or

b) $v_0 = 0$ and $v_1.H^2 > 0$.

\end{definition}
Let us fix a positive class $v \in K_H(X).$

Now, we introduce the \textit{safe area} $U_v$ for $v$ in the following fashion. Set
\[
\Delta_H(E):=(\operatorname{ch}_1(E).H^2)^2
- 2(\operatorname{ch}_2(E).H)(\operatorname{ch}_0(E).H^3).
\]
We consider four possible cases.

\begin{enumerate}

\item If $\Delta_H(v) < 0$, we set $U_v = U$;

\item if $\Delta_H(v) = 0$, we set $U_v = \{(b, w) \in \mathbb R^2 \ | b < \mu_H(\mathbf v)\}$;

\item if $\Delta_H(v) > 0$, we draw a line $l$ which misses $U$. Suppose also that $\operatorname{ch}_0(v) > 0,$ and let $l$ be any line through $\Pi(v)$ that misses $U$. Rotate $l$ clockwise until it intersects the boundary of $U$ in two points: say, $p_1$ and $p_2$. Continuing to move $l$, we finally find the unique line $l = l_v$ so that $b_{p_1} < b_{p_2}$ satisfy the inequality

$$H^3.(b_{p_1} - b_{p_2}) = \operatorname{ch}_1.H^2 - b_{p_2}\operatorname{ch}_0.H^3.$$

The \textit{safe area} $U_v$, in this case consists of points that lie strictly above $l_v$ and strictly to the left of $\Pi(v)$.

\item  if $\Delta_H(v) > 0,$ but $\operatorname{ch}_0(v) = 0$, do the same as in (3) , but let $l$ be any line of the slope $\frac{\operatorname{ch}_2.H}{\operatorname{ch}_1.H^2}$, and move it upwards instead of rotating.
\end{enumerate}

The notion of safe areas will be important for us because of the following statement.

\begin{proposition}\label{FT32} (\cite[Proposition 3.2]{FT}) Suppose that $E \in \mathcal A_b$ is a $\nu_{b, w}$-semistable object of some class $v$ so that $(b, w) \in U_v$. Then,

a) $E$ is a sheaf;

b) for each short exact sequence $0 \to E_1 \to E \to E_2 \to 0$ with $$\nu_{b, w}(E_1) = \nu_{b, w}(E) = \nu_{b, w}(E_2),$$ $E_1$ and $E_2$ are also sheaves. Moreover, $(b, w) \in U_{v_i}$ where $v_i$ is $[E_i]$ for each $i$.
\end{proposition}

As in ~\cite[Section 5]{FT}, let us fix a stability condition $\tau_0$ on the Joyce-Song wall $l_{JS}$, and a stability condition $\tau$ directly above or below it. Then, for any Harder-Narasimhan filtration of any $\tau_0$-semistable object $K$ of the class $v_n = \mathbf{v} - [\mathcal O(-n)]$ by  $\tau$-semistable ones of classes $\alpha_i$, the following holds 

\begin{proposition}\label{FTvn} The set of Harder-Narasimhan factors of $K$ either consists either of a sheaf of class $\mathbf v$ and $\mathcal O(-n)[1]$, or consists  only of  the sheaves of  rank $\leq \operatorname{rk}\mathbf v -1$ in their safe areas, objects of  classes $w_n$ of $\operatorname{rk}(w_n) \leq \operatorname{rk} \mathbf v-2$, and a multiple of $\mathcal O(-n)[1]$.
\end{proposition}

\section{The Mozgovoy-type statement}\label{MozgSec}

\subsection{$\operatorname{Ext}$-vanishing}

Before we start the actual proof, it remains to do the last preliminary: to prove the quasi-smoothness result announced in Subsection~\ref{BridStab}.

Our argument will be completely parallel to the one in \cite[Theorem 7.7]{Mozg}, which, however, works with del Pezzo surfaces, not with Fano threefolds.

The main theorem is as follows.

\begin{thm}\label{Extv} Suppose that an object $F \in \mathcal A_b$ is $\nu_{b, s}$-semistable for $s > \frac{b^2}{2}$, and positive in the sense of Definition \ref{poscl}.
Then, 

a) \begin{equation}
\operatorname{Ext}^{-1}(F, F) = 0
\end{equation}

b) \begin{equation}
\operatorname{Ext}^3(F, F) = 0.
\end{equation}
\end{thm}

\begin{proof}

 For a), one uses \cite[Proposition C.4]{FT}, and~\cite[Sections 4.1-4.2]{PiTo}: this combination shows that $F$ lies in the heart of a certain $t$-structure, and, thus, $\operatorname{Ext}^i(F, F) = 0$ for $i < 0.$

The proof of b) will be somewhat more involved.

\textbf{Step 0.} Let us note that for $F$ either 

$\operatorname{ch}_0^{bH}(F) > 0$; or 

$\operatorname{ch}_0^{bH}(F) = 0$, and $\operatorname{ch}_1^{bH}(F) > 0$; see~\cite[p. 4]{FT}.

\textbf{Step 1.} First of all, we modify our stability condition as in~\cite[(2.2)-(2.3)]{FT0}. Namely, by rescaling and adding a constant, one may replace $\nu_{b, s}$ with $\sigma_{b, w}$ defined by the formula

$$ \sigma_{b, w}(E) = 
    \begin{cases}
      \frac{w\operatorname{ch}_2^{bH}(E).H - \frac{1}{6}\operatorname{ch}_0(E).w^3H^3}{w^2\operatorname{ch}_1^{bH}(E).H^2} \ \text{if} \operatorname{ch}_1^{bH}(E).H^2 \neq 0, \ \\
      +\infty \ \text{if} \operatorname{ch}_1^{bH}(E).H^2 = 0\ \\
    \end{cases}$$

for $w$ satisfying $w = \sqrt{6(s - \frac{b^2}{2})}$; in particular, $w > 0$.

This defines the same Harder-Narasimhan filtrations as $\nu_{b, s}$.

\textbf{Step 2.} Second, we introduce the ``central charge''-function for $\sigma_{b, w}$: for $E \in \mathcal A_b$ with $\operatorname{ch}_i(E) =:u_i$, and $\operatorname{ch}_i^{bH}(E) =: v_i$, we define

\begin{multline}\label{centcharge} Z_{b, w}(E) = \mathbf{i}(w^2v_1.H^2) - (wv_2.H - \frac{w^3}{6}v_0.H^3) \\ = \mathbf{i}(w^2(u_1 - bHu_0).H^2) - (w(u_2-bHu_1 + \frac{b^2H^2u_0}{2}).H - \frac{w^3}{6}u_0.H^3). \end{multline}

Let us fix some particular $(b, w)$, and consider $Z_t := Z_{b-t, w}$, and also the stability condition $\sigma_t := \sigma_{b-t, w}$.

We claim:

\begin{equation}\label{comb}\operatorname{Im}(\frac{d}{dt}Z_t(E) \cdot \overline{Z_t}(E)) \geq 0\end{equation}

for any $\sigma_t$-semistable object $E$.

This will be proved below as the Lemma~\ref{ineq2}.  

\textbf{Step 3.} Let us now deduce the desired statement from the equation~\ref{comb}.

Let $\overline{\mathbb H}$ be the upper halfplane: $$\overline{\mathbb H} = \{z \in \mathbb C \ | \ \operatorname{Im}z \geq 0\}.$$ We claim that, for $t \in [t_0, t_0 + \epsilon)$ and for a very small positive $\epsilon$, $\sigma_t(E) \geq \sigma_{t_0}(E).$

Indeed, it follows from the inequality~\ref{comb} that 

$$Z_t(E) \in Z_{t_0}(E) \cdot \bar{\mathbb H}.$$

Moreover, since our $F$ is positive (see part b) of Lemma~\ref{ineq2}), in fact, the strict inequality holds in equation~\ref{comb}. 

Let $\sigma^+$ stand for the maximal Harder-Narasimhan slope, and $\sigma^-$ for the minimal one. The discussion above shows that the function $t \mapsto \sigma_t^-(F)$ is non-decreasing. Indeed, using that the semistability is a condition open in $t$, we see that each $\sigma_t$-semistable Harder-Narasimhan factor of $F$ has a slope not smaller than $\sigma_{t_0}(F)$.

Thus, we can write: $\sigma_1^{-}(F) > \sigma_0(F).$ 

Similarly, $t \mapsto \sigma_t^+(F)$ is non-decreasing as well. 

Now, by induction,

$$\sigma_{b, w}^-(F(lH)) = \sigma_{b-l, w}^-(F) > \sigma_{b, w}^+(F)$$

for any positive integer $l$. Hence, $\operatorname{Hom}(F, F \otimes K_X) = \operatorname{Hom}(F \otimes K_X^{-1}, F) = 0$, and the desired statement follows by Serre duality.
\end{proof}

\begin{lemma}\label{ineq2} a) \begin{equation}\label{lem}\operatorname{Im}((\frac{d}{dt}Z_t(E)) \cdot \overline{Z_t}(E)) \geq 0\end{equation}

for any $\nu_{b, w}$-semistable $E$ from $\mathcal A_b$.

b) If $E$ is positive in the sense of Theorem~\ref{Extv}, the strict inequality holds in the equation~\ref{lem}.
\end{lemma}

\begin{proof}

a) This  follows from the usual Bogomolov inequality (that is, not the generalized one!).

Namely, as in~\cite[Theorem 7.2]{Mozg}, we have only to prove the statement for $t = 0.$

Then,

$$\frac{d}{dt}Z_t(E)|_{t = 0} = \mathbf{i}(u_0w^2H^3) - (wH^2(u_1 -b H u_0)) = \mathbf{i}v_0w^2H^3 - wH^2v_1, $$

and 

\begin{equation}\label{pos} \operatorname{Im}({Z}_0^{'}(E) \cdot \overline{Z}_0(E)) = w^3((v_1H^2)^2 -(v_0H^3)(v_2H) + w^2v_0^2(\frac{(H^3)^2}{6})). \end{equation}

It is enough to prove the following:

\begin{equation}\label{ineq} (v_1H^2)^2 - (v_0H^3)(v_2H) \geq 0. \end{equation}

This is very similar to the usual Bogomolov-Gieseker inequality. Hence, we use the machinery of \cite[p.34]{BMT}: their argument (by introducing certain functions $f_{a, b}$) shows that one may prove instead that 

\begin{equation}\label{ineq1} f_{-1, 2}(x) := 2(xH^2)^2 - (H^3)(x^2H) \geq 0 \end{equation}

for $x \in \operatorname{NS}_{\mathbb R}(X);$ we explain the details in Remark \ref{posrem} below.

Let us write down: $x = \alpha H + D$, where $D$ is orthogonal to $H$
 under the pairing $\langle \cdot, \cdot \rangle_{H}$ on $\operatorname{NS}_{\mathbb R}(X)$ given by the formula $\langle A, B \rangle_H = A \cdot H \cdot B$. By the (generalized) Hodge index theorem, $D^2H \leq 0$. 

 Now, \begin{align} x^2H = (\alpha H + D)^2H = \alpha^2H^3 + D^2H; \\ xH^2 = \alpha H^3 + DH^2 = \alpha H^3. \end{align}

 Multiply the first by $H^3$, square the second, and subtract:

$$(x^2H)H^3 - (xH^2)^2 = (D^2H)H^3 \leq 0.$$

Thus, 

$$(x^2H)H^3 \leq (xH^2)^2 \leq 2(xH^2)^2.$$

 b) Suppose that $\operatorname{ch}_0(E) > 0$. Then, the third summand on the RHS of the equation~\ref{pos} is strictly positive, and we are done.

For $\operatorname{ch}_0(E) = 0$, only the first summand on the RHS of~\ref{pos} remains, and, by the definition of the positivity, this gives a strictly positive result.

\end{proof}

\begin{remark}\label{posrem}
Here we give a recollection of the result from \cite{BMT} promised above.

For \(a,b\in \mathbb{R}\), we set
\[
f_{a,b}\colon \operatorname{NS}_{\mathbb{Q}}(X)\longrightarrow \mathbb{R}
\]
to be
\[
f_{a,b}(x) := aH^3\cdot x^2H + b(xH^2)^2 .
\]

Recall also that the discriminant \(\Delta(E) \in A^2_{\mathbb{Q}}(X)\) of an object
\(E \in \mathrm{D}^{\mathrm{b}}(X)\) is defined by
\[
\begin{aligned}
\Delta(E)
&= \operatorname{ch}_1(E)^2 - 2 \operatorname{ch}_0(E)\operatorname{ch}_2(E) \\
&= \operatorname{ch}_1^{bH}(E)^2
   - 2 \operatorname{ch}_0^{bH}(E)\operatorname{ch}_2^{bH}(E).
\end{aligned}
\]

The result of \cite[p. 34]{BMT} is as follows:

\begin{theorem}
Take
\(a \in \mathbb{R}_{\geq -1}\) and \(b \in \mathbb{R}\) such that
\(f_{a,b}\) satisfies the following conditions:
\[
\begin{array}{ll}
\mathrm{(a)} & f_{a,b}(x) \geq 0,
\quad \text{for any } x \in \operatorname{NS}_{\mathbb{Q}}(X), \\[4pt]
\mathrm{(b)} & f_{a+1,b}(x) \geq 0,
\quad \text{for any effective class } x \in \operatorname{NS}_{\mathbb{Q}}(X).
\end{array}
\]
Then, for any $\sigma_{b, w}$-semistable object
in $\mathcal A_b$, we have the following inequality:
\begin{equation}
H^3 \cdot H \Delta(E)
+ f_{a,b}\bigl(\operatorname{ch}_1^{bH}(E)\bigr) \geq 0 .
\end{equation}
\end{theorem}

For $a = -1, b = 2$, thus, the statement of equation \ref{ineq}, that is, $$(v_1H^2)^2 - (v_0H^3)(v_2H) \geq 0,$$

would follow from the positivity of $f_{-1, 2}$ and $f_{0, 2}$. The first one was proved above, the second one is trivial.
\end{remark}

\section{The \texorpdfstring{$K$}{K}-theoretic wall-crossing}\label{stablewc}

\subsection{The algorithm}\label{Subsecalg}
In this section, we prove the $K$-theoretic version of Theorem~\ref{mainthm}, and, up to Subsection~\ref{sec: arbcl}, assume that for the given class $\mathbf v \in K(X)$ of $\operatorname{rank} > 0$, the condition $\operatorname{ch}_1(\mathbf v).H^2 = 0$ holds. In Subsection~\ref{sec: arbcl}, we explain how this assumption can be removed.

Our algorithm is, in broad strokes, as follows: working in the class $\mathbf v$ and with the formulas of Theorem~\ref{nonabl}, we cross walls from Gieseker stability to tilt-stability. 

The tilt-stability is known to be equivalent to the $(b, w)$-stability in the large volume limit. From this region, we go down to the Joyce-Song wall. Theorem~\ref{nonabl} at the Joyce-Song wall gives the expression for the corresponding class $\mathbf v$-integral in terms of  the lower rank descendants (in the sense of Subsection~\ref{KThdescend}).

After that, we implement the same procedure one more time, and wall-cross these lower-rank integrals back to Gieseker stability.

\subsection{Stability conditions}\label{sec: stabcond}
Let us, first of all, consider the following two stability conditions for sheaves of class $\mathbf v$. 

The definitions are as follows (we follow~\cite[Section 4]{FT}, though the material is standard). 

To any sheaf $E \in \operatorname{Coh} X$ of  topological type $\alpha \in K(X)$, one canonically associates its Hilbert polynomial $p_{[E]}$:

$$P(\alpha)(t) = \chi(\alpha(t)) = a_dt^d + a_{d-1}t^{d-1} + ... + a_0,$$

where $d \leq 3,$ and $a_d \neq 0$.

The \textit{reduced} Hilbert polynomial $p(\alpha)(t)$ is obtained by dividing by the leading coefficient of $P(\alpha)$:

$$p(\alpha)(t) = \frac{P(\alpha)(t)}{a_d}.$$

Following Joyce~\cite[Section 3.3]{Jo}, we introduce a total order $\prec$ on monic polynomials: $p \prec q$ if either \begin{enumerate} \item $\operatorname{deg} p > \operatorname{deg} q,$ or \item $\operatorname{deg} p = \operatorname{deg} q,$ and $p(t) < q(t)$ for $t >> 0$. \end{enumerate}

The sheaf $E \in \operatorname{Coh} X$ is called Gieseker (semi)stable iff for any short exact sequence of the form $0 \to A \to E \to B \to 0$ in $\operatorname{Coh} X$, one has $p_{[A]} \prec p_{[B]}$ ( $p_{[A]} \ \preccurlyeq p_{[B]}$, respectively).

Unfortunately, the notion of Gieseker stability is not invariant under tensor multiplication by line bundles.

As in~\cite{Rek}, one introduces $D$-Gieseker stability for a divisor $D$.

\begin{definition}
The sheaf $F$ is called $D$-Gieseker stable iff $F(-D)$ is Gieseker stable. This is equivalent to the stability defined using the twisted Hilbert polynomial $p(\alpha)^D(t) := p_{\alpha(-D)}(t).$
\end{definition}

We will denote the moduli space of $D$-Gieseker semistable objects by $\mathcal M_{Gies}^D(\mathbf v).$

The motivation for this (seemingly redundant) definition is the proposition below, see~\cite[Lemma 2.1.11]{Rek}.

\begin{definition} Let us consider the usual $\mu$- and $\nu$-slopes:

$$\mu(F) = \frac{H^{2}\operatorname{ch}_1(F)}{\operatorname{rk} F};  \ \nu(F) = \frac{H\operatorname{ch}_2(F)}{\operatorname{rk} F}.$$

Then, we say $F$ is tilt semistable if for every subsheaf $E \hookrightarrow F$, either

\begin{itemize}
\item $\mu(E) < \mu(F/E)$, or
\item $\mu(E) = \mu(F/E)$, and $\nu(E) \leq \nu(F/E)$

if all these quantities are well-defined.
\end{itemize}
Tilt stable objects are defined by replacing $\leq$ in the \textit{last} inequality by $<$.
\end{definition}

\begin{proposition}\label{Rek0}
$-K_X/2$-Gieseker stability is dominated by tilt-stability.
\end{proposition}

Moreover, as in~\cite[3.2.3]{Rek}, we get

\begin{proposition}\label{Rek} For $w >>0$ and $b < c$, for some constant $c$ depending on the topological type, tilt semistability for a torsion-free $E$ is equivalent to $(b, w)$-semistability.
\end{proposition}

\subsection{Gieseker to $(b, w)$}\label{Giesbw}
We suppose in this section that $\mathbf v$ is either of $\operatorname{rank} > 0$, or corresponds to sheaves of pure $\operatorname{dimension} 2$.
 
For a descendant $\kappa$ and some class $\mathbf w$, we will calculate
\[
\chi(\mathbfcal M_{Gies}(\mathbf w), \kappa \otimes \epsilon).
\]
Using the material of Subsection~\ref{sec: stabcond}, we can equivalently write this as
\[
\chi(\mathbfcal M_{Gies}^{-K_X/2}(\mathbf v), \tau \otimes \epsilon),
\qquad
\mathbf v = \mathbf w \otimes \mathcal O(-K_X/2),
\]
where $\tau$ is the induced descendant.

As in Proposition~\ref{Rek0}, the tilt-stability weakly dominates the $-K_X/2$-Gieseker stability for all $\mathbf v$ as above.

Now, we can use the wall-crossing of Section~\ref{KM} to obtain

\begin{align}\label{Giestilt}
\chi(\mathbfcal M(\mathbf v), \tau \otimes \epsilon^{ti}_{\mathbf v})
&= \chi(\mathbfcal M(\mathbf v), \tau \otimes \epsilon_{Gies}^{-K_X/2}(\mathbf v)) \nonumber\\
&\quad + \sum_k \chi\left(\prod_i \mathbfcal M(\mathbf{v}_{i, k}),
\tau_i\otimes [\boxtimes \epsilon_{Gies}^{-K_X/2}(\mathbf v_{i, k})]\right).
\end{align}
Here further summands correspond to the HN-filtrations of the Gieseker unstable loci on $\mathcal M_{ti}(\mathbf v)$. In particular, all of the classes $\mathbf{v}_{i, k}$ are of some rank $r(\mathbf{v}_{i, k})$, $$0 < r(\mathbf{v}_{i, k}) < r(\mathbf v)$$ ($\operatorname{rank} 0$-semistable summands are prohibited  by the definition of the stability).

Thus, by  Proposition~\ref{Rek}, we are reduced to the calculation of $\chi(\mathcal M_{b, w}(\mathbf v), \epsilon_{b, w}(\mathbf v) \otimes \tau)$ for $w$ large enough and $b$ small enough. As in~\cite{FT}, $w$ just above the Joyce-Song wall from Section~\ref{FTsec} is large enough.

\subsection{The warm up: the  $\operatorname{rank} 2$-situation}\label{rk2wc}

Let $\mathbf v \in K(X)$ be a $\operatorname{rank} 2-$class, $v_n$ be as in Subsection~\ref{FTsec}, and let $n$ be large enough as in Theorem~\ref{pairs}.

We employ an argument a bit different from that of~\cite{FT}.

Let us consider $(b, w)$ just above the Joyce-Song wall $l_{JS}$.

Let ${\underline{\mathcal P}}_n^{(b, w)}(\mathbf v)$ be the moduli stack of Joyce-Song pairs with $(b, w)$-semistable \textit{sheaves} of the corresponding rank. As in~\cite[Appendix B]{FT}, there exists the map $\phi: {\underline{\mathcal P}}_n^{(b, w)}(\mathbf v) \to {\underline{\mathcal M}}_{(b, w)}(v_n)$  which sends a stable pair to its cokernel. The map $\phi$ has a natural derived version $\Phi: \underline{\mathbfcal P}_n^{(b, w)}(\mathbf v) \to {\underline{\mathbfcal M}}_{(b, w)}(v_n)$.

Note that, unlike~\cite{FT}, we use here not the rigidified stack of stable pairs but the non-rigidified one. The relation is as follows: $${\underline{\mathbfcal P}}_n^{(b, w)}(\mathbf v) = \mathbfcal P_n^{(b, w)}(\mathbf v) \times B\mathbb G_m$$ (see~\cite[Eq. (5.24)]{Joy}).

Hence, \cite[Theorem B.2]{FT} can be reformulated as follows.

\begin{proposition}\label{cokernel} $\Phi$ is an isomorphism of  derived stacks.
\end{proposition}

We would like to emphasize that this is the place where one \textit{essentially} uses $\operatorname{rank} 2$-condition. For other ranks, the similar proposition does not hold anymore.

Now, we can describe the $\operatorname{rank} 2$ $\to$ $\operatorname{rank} 1$--wall-crossing as follows.

Let $\kappa$ be a descendant on $\mathcal M_{Gies}(\mathbf w)$. In~\ref{Giesbw}, we reduced the calculation of the integral $\chi(\mathbfcal M_{Gies}(\mathbf v), \kappa \otimes \epsilon)$ to the integral $\chi(\mathbfcal M_{(b, w)}(\mathbf v), \tau \otimes \epsilon)$ for some class $\mathbf v$ of the same rank as $\mathbf w$, and for a canonically defined descendant $\tau$.

For this integral, we use the result of Proposition~\ref{pairs}, and obtain a universal expression in terms of descendants on \begin{enumerate} \item  products of $\operatorname{rank} 1$-tilt-semistable sheaves, and of \item  the moduli space $\mathcal P_n^{(b, w)}(\mathbf v)$. \end{enumerate} 

Note that ~\cite[Appendix B]{FT} proves that stable pairs for the sheaves of class $\mathbf v$ in the sense of tilt-stability, and of $(b, w)$-stability coincide.

The (1)-factors are dealt with analogously to Subsection~\ref{Giesbw} (in fact, easier since their semistable objects are automatically stable).

Let us deal with (2).

Using the map $\Phi$, we rewrite this as $\chi(\mathbfcal M_{(b, w)}(v_n), \psi \otimes \epsilon)$ where $\psi$ is some descendant: analogously to Proposition~\ref{pairs}, one sees that the map $\Phi$ maps descendants to descendants.

Now we are going to wall-cross from $(b, w)$ for $v_n$ to the large volume limit. After that, the goal is to wall-cross back to Gieseker stability.

\subsubsection{The wall-crossing to large volume}

We start with $(b, w)$ as above and go upwards keeping $b < \frac{n}{r-1}$.

Each time we come to a wall of instability for $v_n$, we consider three moduli stacks $\mathcal M_-$, $\mathcal M$, $\mathcal M_+$ of $v_n$-semistable objects directly below the wall, on the wall, and above the wall.

By writing down the wall-crossing from Section~\ref{KM} for the pairs $(\mathcal M, \mathcal M_+)$ and $(\mathcal M, \mathcal M_-)$, we get the formula of the form

$$\chi(\mathbfcal M_-, \psi \otimes \epsilon) = \chi(\mathbfcal M_+, \psi \otimes \epsilon) + \mathcal U$$
where $\mathcal U$ is some expression in terms of the descendants on the products of moduli spaces (which correspond to $\pm$-HN-factors on the wall, as usual).

After performing this process, we get the expression of the form

$$\chi(\mathbfcal M_{(b, w)}(v_n), \psi \otimes \epsilon) = \chi(\mathbfcal M^{tilt}(v_n), \psi \otimes \epsilon) + \sum_i U_{\alpha}\chi(\prod_j \mathbfcal M_{(b_i, w_i)}(\mathbf v_i^j), \psi_i \otimes [\boxtimes \epsilon_j])$$

where $\psi_i$ are some descendants and $U_{\alpha}$ are universal coefficients.

The proof of~\cite[Section 5]{FT} applies here as well and shows that at each wall $(b', w')$ in this wall-crossing process, all $\mathbfcal M_{(b_i, w_i)}(\mathbf v_i^j)$ are moduli spaces consisting of sheaves (that is, not of complexes) either of 
\begin{enumerate}
\item positive classes of $\operatorname{rank} < 2$ in their safe areas, or

\item classes of the form $u_n$ of $\operatorname{rank} = 0$, for some $u$, so that $(b', w')$ lies outside of their Joyce-Song walls.
\end{enumerate}

The moduli spaces of the form (2) are, as in~\cite{FT0}, either empty, or already consist of Gieseker stable rank $0$ pure dimension two sheaves.

Thus, we start performing the same procedure for all products above factorwise using Proposition~\ref{prods}. That is, for each factor of type (1), we are moving upwards throughout its safe area, and each time we consider the $\Theta$-stratification on the product of moduli spaces under consideration induced by the one on this factor.

As in \cite{FT}, the procedure will be finite, though the new products will perhaps arise. Indeed, the reason why the procedure will be finite, is that $\operatorname{ch}_1^{bH}.H^2$ is always greater than $0$, see~\cite[Remark 1.2]{FT}.

The material of Section~\ref{FTsec} is sufficient to conclude that we get in the end the expression of the form

$$\chi(\mathbfcal M_{(b, w)}(\mathbf v), \psi \otimes \epsilon) = \chi(\mathbfcal M^{tilt}(v_n), \psi \otimes \epsilon) + \sum_i V_{\alpha}\chi(\prod_j \mathbfcal M^{tilt}(\mathbf v_i^j), \psi_i\otimes [\boxtimes \epsilon_j])$$

where $\mathbf v_i^j$ are either $\operatorname{rank} 1$ classes, or correspond to the $\operatorname{rank} 0$-sheaves of pure dimension $2$, and $V_{\alpha}$ are universal coefficients.

We apply now factorwise the procedure from Subsection~\ref{sec: stabcond} to obtain the formula 
$$\chi(\mathbfcal M_{(b, w)}(\mathbf v), \psi \otimes \epsilon) = \chi(\mathbfcal M^{Gies}(v_n), \psi \otimes \epsilon) + \sum_i W_{\alpha}\chi(\prod_j \mathbfcal M^{Gies}(\mathbf v_i^j), \psi_i \otimes [\boxtimes \epsilon_j])$$
for various $\mathbf v_i^j$ as above and universal coefficients $W_{\alpha}$.

The only thing to be explained now is how to express $\operatorname{rank} 1$-contributions through the $\operatorname{rank} 0$-ones.

\subsubsection{The $\operatorname{rank} 1 \to \operatorname{rank} 0$ step.}

Here one employs the procedure completely analogous to what was said in Subsection~\ref{rk2wc}. That is, one uses the fact that the map of taking cokernel, $${\underline{\mathcal P}}_n^{tilt}(\mathbf v) \to {\underline{\mathcal M}}_{(b, w)}(v_n),$$ for a $\operatorname{rank} 1$-class $\mathbf v$, is an isomorphism for $(b, w)$ just above the Joyce-Song wall.

Moreover, the argument is even easier. Indeed, for the $\operatorname{rank} 0$-sheaves of class $v_n$ there are no other walls, and they are all (\cite{FT0})  Gieseker stable rank $0$ pure dimension two sheaves.

\subsection{The Joyce-Song wall in general}\label{JSw}
This is the central subsection of the present paper. It is designed to prove Theorem~\ref{Kthmain}.

Since the analogue of Proposition~\ref{cokernel} is no more available (see~\cite[Theorem A.2]{FT} for an approximation), one has to try to repeat the original argument of~\cite{FT}.

The algorithm is as before: first, to start with Gieseker for the class $\mathbf v$, and move right above the Joyce-Song wall. This is as above, and is independent of rank.

Now we write the wall-crossing of Section~\ref{KM} for $l_{JS}$.

We consider the stability condition $(b, w)$ on $l_{JS}$, and the stability conditions $(b^+, w^+)$ and $(b^-, w^-)$ immediately above it and below it, and the corresponding moduli stacks $\mathcal M^{\pm, 0}(v) := \mathcal M^{ss}_{b^{\pm, 0}, w^{\pm, 0}}(v)$. Using~\cite[Theorem 3.4(2)]{FT} for any descendant $\tau$ on $\mathbfcal M_{(b, w)}(v_n)$, we get:

\begin{align}\label{rk2FT}\chi(\mathbfcal M^+(v_n), \tau \otimes \epsilon)  = \chi(\mathbfcal M^-(v_n), \tau \otimes \epsilon) + \chi(\mathbfcal  M^0(\mathbf v), \tau' \otimes \epsilon) + \mathcal U\end{align}

where $\mathcal U$ is some expression in terms of the descendants on the products of the lower rank moduli spaces.

More precisely, according to Section~\ref{KM} (see also the proofs in~\cite{KM}), the terms of this difference correspond to the following types $\tau_{\pm}$-HN-strata on ${\underline{\mathbfcal M}}^0(v_n)$:

\begin{enumerate}
\item The first term on the RHS is evident.

\item The second term on the RHS of equation (\ref{rk2FT}) is for the Joyce-Song stratum (whose center $\mathcal M^+(\mathbf v) \times B\mathbb G_m$ consists of the objects of the form $\mathcal F \oplus O(-n)[1]$ where $\mathcal F$ is $(b^+, w^+)$-stable of class $\mathbf v$); $E_+, E_-$ are the corresponding complexes of the form (\ref{Ea}) above/below the wall, respectively. Then, 
\begin{itemize}
\item $\alpha$) the universal wall-crossing coefficient $\widetilde{U}$ from Section~\ref{KM} corresponding to this term is equal to $\pm 1$: this is a result of the fact that, according to~\cite{KM}, our wall-crossing coefficients are the same as in motivic wall-crossing, and the calculation in~\cite[Eqns. (71)-(72)]{FT}.

\item $\beta$) it follows from our discussion that, modulo this sign, the precise recipe for calculating $\tau'$ is as follows: it is a descendant so that $$\chi(\mathbfcal M(\mathbf v), \tau' \otimes \epsilon_{\mathbf v}) = \chi(\mathbfcal M(v_n), \tau \otimes [\epsilon_{\mathbf v}, \epsilon_{[\mathcal O(-n)[1]]}]).$$ The last formula makes sense since $\mathbfcal M(\mathcal O(-n)[1])$ consists of a point, so a descendant on $\mathbfcal M(\mathbf v) \times \mathbfcal M(\mathcal O(-n)[1])$ can be interpreted as a descendant on $\mathbfcal M(\mathbf v).$
\end{itemize}

\item Other $HN$-strata correspond to the products of the form

\begin{equation}\label{centerJS}\left(\prod {\underline{\mathbfcal M}}(\mathbf v_i)\right) \times \left(\prod {\underline{\mathbfcal M}}(\mathbf w_i)\right) \times B\mathbb G_m^k\end{equation}

by the results of Subsection~\ref{FT32}: here $b \leq \frac{n}{r-1}$, and
\begin{itemize}
\item i) $\mathbf v_i$ are positive classes of rank $<r$ in their safe areas;

\item ii) $\mathbf w_i$ are the classes of the form $u_n:=u-[\mathcal O(-n)]$ (for some $u$'s) of rank $\leq \operatorname{rk}(\mathbf v) - 2$;

\item iii) $B\mathbb G_m$ factors correspond to the $\mathcal O(-n)[1]$ summands.
\end{itemize}
\end{enumerate}

\subsubsection{Two goals.}
The discussion above shows that we now have two goals: 

1. To explain that the ``lower rank''-terms in (1), (2), (3).i), (3).ii), and (3).iii) above can indeed be wall-crossed to descendant integrals over moduli spaces of lower rank Gieseker sheaves.

2. To prove that the calculation of any descendant integral on $\mathcal M(\mathbf v)$ amounts to a calculation of one corresponding to a descendant of the form $\tau'$ for some $\tau.$

\subsubsection{Lower rank wall-crossing} We  start with the first of these two goals, and are doing factor-wise wall-crossing as in the $\operatorname{rank} 2$-case.

\begin{enumerate} \item The (3).iii)-factors do not contribute at all (they become points after rigidification). 

\item The (3).i)-factors are dealt with as in the $\operatorname{rank} 2$-case, that is, just wall-crossed to the large volume (see~\cite[Proposition 5.1]{FT}) inductively.

\item The less trivial part is (3).ii) and (2), which both concern classes of the form $u_n$ with $n$ large enough \textit{relatively to $u$}, as in~\cite{FT}. There are two possibilities in the case of each of these classes: we are either above its Joyce-Song wall, or below it. 

\begin{itemize} \item In the first case, we wall-cross to the large volume staying away from JS-wall, while keeping $b \leq \frac{n}{r-1}.$

\item In the second case, we wall-cross \textit{downwards} similarly. The argument in \cite{FT} explains that this induction goes through. The reason is that, staying outside of the JS-wall, we will meet only HN-factors of types (3).i), (3).ii), (3).iii), and of lower rank.

(In the low volume chamber the invariants are, as in~\cite{FT}, zero.)
\end{itemize}

This is completely analogous to~\cite[Eqn. (77)]{FT}.
\end{enumerate}
\subsubsection{Descendants }\label{goal2} For our second goal, we have to study the operation $\tau \to \tau'$ described above.

First of all,  since all the moduli spaces involved at this stage are moduli spaces of sheaves (and not \textit{complexes}), the results of~\cite{Joy} endow them with \textit{cohomological} fundamental classes, as we explain in~\cite{KM}.

Thus, the results of Section~\ref{KM} and Section~\ref{cohtrans} below explain that any integral $\chi(\mathcal M(\mathbf v), \kappa \otimes \epsilon)$ over a moduli space of tilt-semistable sheaves can be reduced to the calculation of $\int_{[\mathcal M]^{vir}}\upsilon(\kappa)$ where $[\mathcal M]^{vir}$ is Joyce's virtual fundamental class in \textit{homology}, and $\upsilon(\kappa)$ is some cohomological descendant which depends on $\kappa$ formally. The same is true \textit{vice versa}.

Moreover, as it is clear from Appendix \ref{Appendix}, the operation $\tau \mapsto \tau'$ has a natural cohomological counterpart which we will denote $\psi \to \psi'$, by abuse of notation. It is clear that $\upsilon(\kappa)' = \upsilon(\kappa').$

Thus, it would be enough to prove that any cohomological descendant on $\mathcal M(\mathbf v)$ can be written down as $\upsilon'$ for some $\upsilon$. This is done in Appendix \ref{Appendix}.

\subsection{The end of the proof}\label{sec: arbcl}

We explain how to remove the assumption $\operatorname{ch}_1.H^2 = 0$ from Subsection~\ref{Subsecalg} following~\cite[Section 5]{FT}.

To do this, the authors of \cite{FT} replace $\operatorname{ch}$ by $\operatorname{ch}^{tH} = e^{-tH}\operatorname{ch}$, and argue that, for $b^t = b - t$, and $w^t = w - bt + \frac{1}{2}t^2$, $\nu_{b, w}$-stability for the given Chern character $\operatorname{ch}$ is equivalent to $\nu_{b^t, w^t}$ for $\operatorname{ch}^{tH}$. Moreover, for $t = \operatorname{ch}_1(\mathbf v).H^2/rH^3$, $$\operatorname{ch}_1^{tH}.H^2 = 0.$$

Then, they just replace $\operatorname{ch}$ in their proof by $\operatorname{ch}^{tH}.$

Now, note that twisting by $e^{\pm tH}$ preserves positivity (in the sense of Theorem~\ref{Extv}), and that $w^t > \frac{(b^t)^2}{2}$ is equivalent to $w > \frac{b^2}{2}$ so that there are no issues with applying Theorem~\ref{Extv}.

Thus, we can apply the described method in our context as well.

\section{Transition to cohomology}\label{cohtrans}

\subsection{}
The next theorem uses the notation of Subsection~\ref{KThdescend}.  We suppose that $\mathcal M :=\mathbfcal M(v) $ is the derived rigidified moduli stack of objects in some abelian subcategory of the derived category of coherent sheaves that admits Joyce's virtual classes, and, in particular, framing functors.

\begin{thm}\label{mainredcoh} a) Each integral of the form $\chi(\mathcal M, \tau \otimes \epsilon)$ can be written as a universal expression in terms of various $\int\limits_{[\mathcal M]^{vir}} \psi =: I(\psi)$ for various $\operatorname{weight} 0$-cohomological descendants $\psi$.

b) Each integral of the form $I(\psi)$ for any cohomological $\operatorname{weight} 0$-descendant $\psi$ can be written as a universal expression in terms of $\chi(\mathcal M, \tau \otimes  \epsilon)$.
\end{thm}

Before we start the proof, let us recall the definition of Adams operations that will be useful for b) below. \footnote{For the relevant reference, see~\cite[Subsection 3.5.2]{MS}.}

\begin{definition}\label{def: Adams} The \textit{Adams operations} are a sequence of operations $$\psi^N: K(\mathcal M) \to K(\mathcal M), \ N \in \mathbb N,$$ which are defined by the inductive formula

$$\psi^N - \psi^{N-1} \otimes \lambda^1 + \ldots + (-1)^{N-1}\psi^1 \otimes \lambda^{N-1} + (-1)^NN\lambda^N = 0$$

where $\lambda$ stands for the exterior power operation on $K(\mathcal M)$.
\end{definition}

\begin{remark}\label{rem: Adams}
 The operations $\psi^N$ satisfy the identities of the form

$$\operatorname{ch}_k(\psi^N(\mathcal U)) = N^k\operatorname{ch}_k(\mathcal U).$$

For schemes, this follows, e.g. from~\cite[Theorem 3.1]{KR}. For stacks, this follows from the material of~\cite[Section 5.2]{JoyceRingel} (the same reference is for the Chern character map on the perfect K-theory of algebraic stacks).
\end{remark}

\begin{proof}
a) 

\textbf{Step 1.} Here, we will use Section~\ref{KM}. What we obtain by applying it formally is 

$$\chi(\mathcal M, \tau \otimes \epsilon) = \int_{[\mathcal M]^{vir}} \operatorname{ch}(\tau) \cup \operatorname{td}(T_{\mathcal M}^{vir})$$

where $T_{\mathcal M}^{vir}$ is the virtual tangent complex of $\mathcal M$. We will also use the shorthand $T^{vir}$ for the virtual tangent complex of $\underline{\mathcal M}$.

Let $d$ stand for the virtual dimension of $\mathcal M$, that is, suppose that $[\mathcal M]^{vir} \in H_{2d}(\mathcal M)$. Note that the class $[\operatorname{Td}(T_{\mathcal M}^{vir})]^{\leq d}$ (that is, part of $\operatorname{Td}(T^{vir})$ in the cohomological degrees $\leq d$) can be written in terms of the cohomological descendants. This follows from the formula (see~\cite[Example 2.6]{BLM})

$$\operatorname{ch}(T^{vir}) = - \sum\limits_{i, j \geq 0} \sum_t (-1)^{i-p_L^t + \operatorname{dim}(X)}\pi_{1*}(\operatorname{ch}_i(\mathcal U)\cup \gamma_t^L)\pi_{1*}(\operatorname{ch}_j(\mathcal U)\cup \gamma_t^R)$$

for the K{\"u}nneth decomposition $$\Delta_*(\operatorname{td}(X)) = \sum \gamma_t^L \boxtimes \gamma_t^R,$$ $\gamma_t^i \in H^{(p_i^t, q_i^t)}(X), \  i \in \{L, R\}$, and the universal complex $\mathcal U$ over $\mathcal M \times X$.

\begin{remark} Here we work in the notation of Subsection~\ref{KThdescend}, and $\pi_{1*}$ stands for the \textit{cohomological} pushforward. We will denote the $K$-theoretic pushforward by $\pi_{1*}^K$ in the present proof.
\end{remark}

\textbf{Step 2.} We are reduced to showing that each $\operatorname{ch}_i(\tau)$ is a cohomological descendant on $\underline{\mathcal M}$. Suppose that $\tau = \pi_{1*}^K(\Psi \otimes \prod_i\pi_2^*\gamma_i)$ for the appropriate combination $\Psi$ of Schur functors as in~\ref{descform}. Then,

$$\operatorname{ch}_i(\tau) = \operatorname{ch}_i(\pi^K_{1*}(\Psi \otimes \prod_i \pi_{2}^*\gamma_i)) = [(\pi_{1*}(\operatorname{ch}(\Psi \otimes \prod_i\pi_2^*\gamma_i) \cup \operatorname{Td}(X^n)))]^i,$$

by the Grothendieck-Riemann-Roch formula.

Thus, we have to prove that $\pi_{1*}(\operatorname{ch}_i(\Psi) \cup \prod_i\pi_2^*\epsilon_i)$, $\epsilon_i = \operatorname{ch}(\gamma_i)$, is a cohomological descendant. For this, note, first of all, that the following lemma holds.

\begin{lemma}  $\pi_{1*}(\operatorname{ch}_i(\Psi) \cup \prod_i\pi_2^*\epsilon_i)$ has the form $\pi_{1*}(P(\{\operatorname{ch}_j(\mathcal U)\}_{j \in \mathbb N}) \cup \prod_i\pi_2^*\epsilon_i\})$ for some polynomial $P$.
\end{lemma} 

\begin{proof} Our main concern here is the factor corresponding to the $\operatorname{Ext}(-, -)$-factors in the formula~\ref{descform}. However, the way of its treatment is clear from the formula of ~\cite[Proposition 5.2]{Gross}, (see also~\cite[Theorem 4.7]{BLM}) which says, in particular, that, over $\underline{\mathcal M}\times  X \times X$,

\begin{align}\operatorname{ch}(\pi_{1*}\operatorname{Hom}(\pi_{1, 2}^*\mathcal U_1, \pi_{1, 3}^*\mathcal U_2)) \\= \sum_{j, k \geq 0, \  \  v, w \in Q} c_{v, w}\pi_{1*}(\operatorname{ch}_j\mathcal U \cup \pi_2^*v) \cup \pi_{1*}(\operatorname{ch}_k \mathcal U \cup \pi_2^*w)\end{align}

which is considered as an equality in the topological $K$-theory with the abuse of notation $X = X^{an}$, for some basis $Q$ in $K^0(X^{an})_{\mathbb Q} \oplus K^1(X^{an})_{\mathbb Q},$ and for constants $c_{v, w}.$

\begin{remark} Here, and only here, in the present paper it is important that $X$ is \textit{in class $D$} (since Gross assumes it in his Proposition 5.2): this means that the semi-topological $K$-theory of $X$ in the sense of Friedlander-Walker coincides with its topological $K$-theory. However, it turns out that this condition \textit{always} holds for Fano varieties by a result of Voineagu, that is, by the main theorem of \cite{Voi}.
\end{remark}
\end{proof}

Thus, to finish the proof of a), we just have to show that each class of the form $$\pi_{1*}((\prod_{k=1}^{n} \operatorname{ch}_{i_k}(\mathcal U^{j_k})) \cup \prod_{i=1}^m\pi_2^{i,*}\epsilon_i)$$ can be rewritten as $P(\{\operatorname{ch}_i(\kappa_j)\}_{j \in J})$ for some set $\kappa_j$ of cohomology classes on $X$.

\textbf{Step 3.} To prove the latter statement, we illustrate it in the simplest possible situation, and the general case follows analogously.

Suppose that we are dealing, on $\underline{\mathcal M}\times X \times X$ with $$\tau := \pi_{1*}(\operatorname{ch}_i(\mathcal U_1)\cup \operatorname{ch}_j(\mathcal U_1) \cup \pi_2^{1*}\gamma\cup \pi_2^{2*}\psi).$$

Then, using $\underline{\mathcal M} \times \underline{\mathcal M} \times X \times X$ and the projection formula, $\tau$ can be written as  

$$\pi_{1*}(\operatorname{ch}_i(\mathcal U) \cup \pi_2^{*}\gamma) \cup \pi_{1*}(\operatorname{ch}_i(\mathcal U) \cup \pi_2^{*}\psi).$$

This uses the Cartesian diagram 

\[\begin{tikzcd}
	{\underline{\mathcal M} \times X^2} && {(\underline{\mathcal M} \times X)^2} \\
	{\underline{\mathcal M}} && {\underline{\mathcal M}^2.}
	\arrow["{\Delta_{\underline{\mathcal M}} \times \operatorname{id}}", from=1-1, to=1-3]
	\arrow[from=1-1, to=2-1]
	\arrow[from=1-3, to=2-3]
	\arrow["{\Delta_{\underline{\mathcal M}}}", from=2-1, to=2-3]
\end{tikzcd}\]

We are done with a).

b)  What we will be using is the following statement which will be proved in Step 2 below:

\textbf{Claim.} \textit{Let $d$ be the complex virtual dimension of $\mathcal M$ (so that $[\mathcal M]^{vir} \in H_{d}(\mathcal M)$). Then, for any cohomological descendant $D$, there exists a $K$-theoretic descendant $\tau$ so that $\operatorname{ch}(\tau) = D \operatorname{mod} H^{> d}(\mathcal M)$.}

\textbf{Step 1.} Let us, first of all, deduce the statement of b) from the claim. We have to show that, for any cohomological descendant $D$, there exists a $K$-theoretic descendant $\tau$ so that $[\operatorname{ch}(\tau)\operatorname{td}(T^{vir})]^d = D$. Now, $\operatorname{Td}(T^{vir})$ is formally invertible, and can be rewritten in terms of the descendants, as it was discussed in the part a) above.

Thus, we are done using the claim for the cohomological descendant $D(\operatorname{Td}^{-1}(T^{vir})_{\leq d})$.

\textbf{Step 2a.} Now we are going to prove the Claim. It is enough to prove it on $\underline{\mathcal M}$ instead of $\mathcal M$ itself. 

Let us, first of all, show that for  the universal sheaf $\mathcal U$ over $\underline{\mathcal M} \times X$, $\operatorname{ch}_i(\mathcal U) \cup \pi_2^*(\gamma)$ can be written in terms of $\operatorname{ch}(\psi^N\mathcal U \otimes \pi_2^*(\Gamma))$ modulo $H_{>d}(\mathcal M \times X)$ for $\psi$ as in Definition~\ref{def: Adams}, and for various classes $\Gamma \in K(X)$.

First of all, $\gamma$ is irrelevant here, and we can set $\gamma = 1$ without loss of generality. Indeed, we recall that $K(X)$ stands here for Blanc's topological K-theory, so the Chern character map $K(X) \to H^*(X)$ is surjective.

Second, by the Remark~\ref{rem: Adams}, $\operatorname{ch}_i(\psi^N(\mathcal U)) = N^i\operatorname{ch}_i(\mathcal U)$. Hence, we can write the $(d +1) \times (d+1)$ matrix $A$ so that $A_{i, j} = \operatorname{ch}_j(\psi^i(\mathcal U))$, and, noting that the Vandermonde determinant is non-zero, we are done.

We now return to the class $\gamma$, and will write $\mathcal U_{i}(\gamma)$ for the resulting expression:
$$\operatorname{ch}_i(\mathcal U) \cup \pi_2^*(\gamma) = [\operatorname{ch}(\mathcal U_{i}(\gamma))]^i$$
where the superscript $i$ stands for the degree $i$-part of the given class.

\textbf{Step 2b.} We now prove the Claim for $$D = \pi_{1*}(\operatorname{ch}_i(\mathcal U) \cup \pi_2^*(\gamma)).$$

Let us consider $\pi^K_{1*}(\mathcal U_{i}(\gamma \cup \operatorname{Td}(X)^{-1}))$. Its Chern character is equal to $D$ by the Grothendieck-Riemann-Roch formula:
\begin{align} \operatorname{ch}(\pi^K_{1*}(\mathcal U_{i}(\gamma \cup \operatorname{Td}(X)^{-1}))) = \pi_{1*}(\operatorname{ch}_i(\mathcal U) \cup \pi_2^*(\gamma \cup \operatorname{Td}(X) \cup \operatorname{Td}(X)^{-1})) \\ = D
\end{align}

where all the equalities are $\operatorname{mod} H^{>d}(\underline{\mathcal M})$.

Finally, $\operatorname{Td}(X)$ lies in $\operatorname{Im}(\operatorname{ch}: K(X) \to H^*(X))$ since its graded components may be expressed in terms of $\operatorname{ch}_i(T_X)$ so one may repeat the above argument with the Adams operations.

\textbf{Step 2c.} For more complicated descendants $D$, that is, for $D$ of the form $\prod_{k=1}^n \operatorname{ch}_{i_k}(\gamma_k)$, one now argues similarly but using the commutative diagram above to deal with the multiplication, as in part a).

\end{proof}

To prove Theorem~\ref{mainthm}, and hence finish the present paper, it remains to record the following corollary.

\begin{corollary}\label{multiplefactors}
For any classes $\mathbf v_1,$ $\mathbf v_2$, ... $\mathbf v_n$ in $K(X)$ let $\mathcal M := \prod_{i=1}^n\mathcal M_i := \prod_{i=1}^n \mathcal M(\mathbf v_i)$.

Then, the descendant integrals  over $\mathbfcal M$, against  the product of $\epsilon$-classes, can be written as universal polynomials in terms of $\int_{[\mathcal M_i]^{vir}}\psi$ for various cohomological descendants $\psi$, and \textit{vice versa}.
\end{corollary}
\begin{proof}
This is fully analogous to Theorem~\ref{mainredcoh}.
\end{proof}

\begin{remark} Note that, from this statement, together with Theorem~\ref{Kthmain}, Theorem~\ref{mainthm} immediately follows.
\end{remark}

\appendix

\section{Injectivity of vertex operators}\label{Appendix}
\begin{center}
  {\scshape Ivan Karpov and Miguel Moreira}
\end{center}
\medskip

 \subsection{A basis for the algebra of symmetric functions}

Let $\Lambda=\mathbb Q[p_1, p_2, \ldots]$ be the algebra of symmetric functions in infinitely many variables, where $p_i$ is regarded as the $i$-th power sum polynomial. The algebra $\Lambda$ has a natural grading given by $\deg(p_i)=i$; we denote by $\Lambda_d$ the degree $d$ part of $\Lambda$. The elements $p_i$, for $i\geq 1$ act on $\Lambda$ by multiplication. We define the action of $p_{-i}$ on $\Lambda$ by
\[p_{-i}=i\frac{\partial}{\partial p_i}\,,\quad i\geq 1\,.\]

Let us define the vertex operators $\{W_n\}_{n\in \mathbb Z}$ by
\[\sum_{n\in \mathbb Z}W_n z^n=\exp\left(\sum_{n>0}\frac{p_n}{n}z^{n}\right)\exp\left(\sum_{n<0}\frac{2p_n}{n}z^{n}\right)\,.\]
Note that the operator $W_n$ has degree $n$. Such operators arise naturally from the vertex algebra associated to the lattice $(\mathbb Z, 2)$. Given a partition $\lambda$
\[\lambda_1\geq \lambda_2\geq \ldots \lambda_\ell>0\]
of length $\ell=\ell(\lambda)$ and size $d=\lambda_1+\ldots+\lambda_l$ we define the symmetric function
\[t_\lambda=W_{\lambda_1}\ldots W_{\lambda_\ell}(1)\in \Lambda_d\,.\]
It is well-known (\cite{Jing}) that Schur functions are obtained if one modifies the above definition by removing the 2 from the definition of $W_n$.

\begin{proposition}\label{prop: basistlambda}
    The elements $t_\lambda$, with $\lambda$ ranging through the partitions of $d$, form a basis of $\Lambda_d$. 
\end{proposition}
\begin{proof}
Let $\mathbb Q[\mathbb Z]$ be the group algebra of $\mathbb Z$, and write $\{e^m\}_{m\in \mathbb Z}$ for its basis. The operators $X_{-n}\colon \mathbb Q[\mathbb Z]\otimes \Lambda\to \mathbb Q[\mathbb Z]\otimes \Lambda$ (after setting $\alpha=1$) in \cite{Caijing} are related to our $W_n$ by 
\[X_{-n}(e^m\otimes f)= e^{m+1}\otimes W_{n-2m-1}(f)\,.\]
Thus
\[e^{\ell(\lambda)}\otimes t_\lambda=X_{-(\lambda_1+2\ell-1)}X_{-(\lambda_2+2\ell-3)}\ldots X_{-(\lambda_{\ell-1}+3)}X_{-(\lambda_{\ell}+1)}(e^0\otimes 1)\, \]
and by \cite[Theorem 3.6]{Caijing} we conclude that we can write
\[t_\lambda=\sum_{\mu\geq \lambda} a_{\lambda\mu}h_\mu\]
where $a_{\lambda\mu}\in \mathbb Q$ are such that $a_{\lambda\lambda}=1$, $\mu\geq \lambda$ denotes the dominance relation and $h_\mu=h_{\mu_1}\ldots h_{\mu_k}$ where $h_i$ is the $i$-th complete homogeneous symmetric polynomial. Since $\{h_\lambda\}$ is a basis of $\Lambda$, it follows that $\{t_\lambda\}$ is a basis as well.
\end{proof}

\begin{corollary}\label{cor: injectivesymmfunctions}
    If $n\geq d\geq 0$ then the operator
    \[W_n\colon \Lambda_d\to \Lambda_{d+n}\]
    is injective.
\end{corollary}
\begin{proof}
    This is clear from Proposition \ref{prop: basistlambda} since, by definition, $W_n(t_\lambda)=t_{\{n\}\cup \lambda}$ where $\{n\}\cup \lambda$ is the partition obtained by adding $n$ to $\lambda$; note that $n\geq d$ implies that $n \geq \lambda_1$.
\end{proof}
\begin{remark}
    Numerical evidence suggests that $W_n$ is actually injective for any $n\geq 0$ and no restriction on $d$. It is also apparent that the change of basis matrix between $\{t_\lambda\}$ and the basis of Schur functions $\{s_\lambda\}$ is quite simple. 
\end{remark}

\subsection{Injectivity on vertex algebras}

Let us now recall the notion of lattice vertex algebras \cite[Section 5.4, 5.5]{Kac}. For this, we let $L\simeq \mathbb Z^n$ be a finitely generated free abelian group and $Q\colon L\otimes L\to \mathbb Z$ a symmetric integral pairing; we do not require $Q$ to be non-degenerate. The Fock space associated to $L$ is 
$F_L=\textup{Sym}\big(L\otimes_\mathbb Z  t^{-1}\mathbb Q[t^{-1}]\big)$.
We assign a grading to $F_L$ by letting $t$ have degree $-1$, and write $F_L^d$ for the degree $d$ part of $F_L$. We write $v_n$ for $v\otimes t^n$, when $n<0$. The Heisenberg Lie algebra is spanned by $\{v_n\}_{n\in \mathbb Z}$ and $\id$, and it has Lie bracket
\[[v_n, u_m]=n\delta_{n+m=0}Q(v,u)\id\,.\]
The Heisenberg algebra acts naturally on the Fock space: the creation operators $v_n$ with $n<0$ act by multiplication, and $v_n$ with $n\geq 0$ are declared to annihilate $1\in F_V$.

The lattice vertex algebra ${\bf V}(L)$ associated to $(L, Q)$ has underlying vector space given by
\[{\bf V}(L)=\bigoplus_{\alpha\in L} e^\alpha\otimes F_{L}\,.\]
By construction, given $\alpha\in L$, the restriction of the vertex operators
\begin{equation}\label{eq: operatorealpha}
e^\alpha_k\colon F_L\simeq e^\beta\otimes F_{L}\to e^{\alpha+\beta}\otimes F_L\simeq F_L\end{equation}
associated to $e^\alpha=e^\alpha\otimes 1$ are given by
\[\sum_{k\in \mathbb Z}e^\alpha_k z^{-k-1}=\epsilon_{\alpha, \beta} z^{Q(\alpha, \beta)}\exp\left(-\sum_{n<0}\frac{v_n}{n}z^{-n}\right)\exp\left(-\sum_{n>0}\frac{v_n}{n}z^{-n}\right)\]
where $\epsilon_{\alpha, \beta}\in \{1,-1\}$.

\begin{proposition}\label{prop: injectivityVA}
    Let $u=e^\alpha\otimes 1\in {\bf V}(L)$ and $v=e^\beta\otimes m\in {\bf V}(L)$ with $\alpha, \beta\in L$ and $m\in F_L^d$. Suppose that $Q(\alpha, \alpha)=2$ and that $d\leq -(k+1+Q(\alpha, \beta))$. Then $u_k v=0$ implies that $v=0$.
\end{proposition}
\begin{proof}
    Let $V$ be the orthogonal complement of $\alpha$ in $L\otimes \mathbb Q$ with respect to the pairing $Q$, so that $L\otimes \mathbb Q=\mathbb Q\alpha\oplus V$. Then $F_L\simeq F_\alpha\otimes F_{V}$. If we identify $F_\alpha$ with the algebra of symmetric functions $\Lambda$ via $p_i=\alpha_{-i}$ then the action of $\alpha_i$ in $F_\alpha\simeq \Lambda$ is identified with $2p_{-i}$ for $i>0$, since $Q(\alpha, \alpha)=2$. On the other hand, all the operators $\alpha_i$ and $e_k^\alpha$ commute with $v_{n}$ for $v\in V$, since $Q(\alpha, v)=0$. Hence the operator $e_{k}^\alpha$ restricted to \eqref{eq: operatorealpha} is identified with $W_{-(k+1+Q(\alpha, \beta))}\otimes \textup{id}_{F_V}$, so the conclusion follows from Corollary \ref{cor: injectivesymmfunctions}.
\end{proof}

\subsection{Injectivity on the Lie algebra}

Given any vertex algebra ${\bf V}$, one associates (see \cite{Bor}) a Lie algebra  ${\bf V}/\operatorname{im}(T)$ with Lie bracket defined by
\[[\overline u, \overline v]=\overline {u_0 v}\,\]
where $\overline u$ denotes the image of $u\in {\bf V}$ in the quotient ${\bf V}/\operatorname{im}(T)$. The following is the main result of this appendix.

\begin{proposition}\label{prop: injectivity bracket}
    Let $u=e^\alpha\otimes 1\in {\bf V}(L)$ and $v=e^\beta\otimes m\in {\bf V}(L)$ with $m\in F_L^d$. Suppose that $Q(\alpha, \alpha)=2$, that $d\leq -Q(\alpha, \beta)-1$, and furthermore that there exists $\gamma\in L$ such that
    \[Q(\alpha, \gamma)=0\textup{ and }Q(\beta, \gamma)\neq 0\,.\]
    Then 
    \[[\overline u, \overline v]=0\Rightarrow \overline v=0\,.\]
\end{proposition}
\begin{proof}
The proof combines the existence of $\gamma$-normalized lifts from the Lie algebra to the vertex algebra, which is discussed in a different language in \cite[Section 1.2.1]{KLMP}, and the proof of \cite[Lemma 5.11]{BLM}. We spell out the details here in a purely vertex algebra language.

First of all, we claim that $Q(\beta, \gamma)\neq 0$ implies that we can find a (unique) lift $v'\in {\bf V}={\bf V}(L)$ of $\overline v$ with the property that $\gamma_{1}v'=0$. By scaling $\gamma$ (possibly allowing it to be in $L\otimes \mathbb Q$) we assume that $Q(\beta, \gamma)=1$, and let
\[v'=\sum_{j\geq 0}\frac{(-1)^j}{j!}T^j \gamma_1^j u\,.\]
Here, $T^j$ and $\gamma_1^j$ stand for $T\circ\ldots \circ T$ and $\gamma_1\circ \ldots \circ \gamma_1$. Clearly we have $\overline v'=\overline v$. Using that $[T, \gamma_1]=\gamma_0$ and $\gamma_0$ acts on $e^\beta\otimes F_L$ as multiplication by $Q(\beta, \gamma)=1$, we obtain
\[\gamma_1 v'=\sum_{j\geq 0}\frac{(-1)^j}{j!}[\gamma_1, T^j] \gamma_1^j u+\sum_{j\geq 0}\frac{(-1)^j}{j!}T^j \gamma_1^{j+1} u=0\,.\]
Let $w=u_0 v'$. Since we have $\gamma_1u=0$ (trivially) and $\gamma_0 u=0$ (since $Q(\alpha, \gamma)=0$), using \cite[Section 4 (v)]{Bor} we have
\[\gamma_1 w=u_0(\gamma_1 v)+(\gamma_0 u)_1 v-(\gamma_1 u)_0 v=0\,.\]
By hypothesis, we have 
\[\overline w=[\overline u, \overline v']=[\overline u, \overline v]=0\]
so $w=T(w')$ for some $w'$. But then
\begin{align*}w&=\sum_{j\geq 0}\frac{(-1)^j}{j!}T^j \gamma_1^j w=\sum_{j\geq 0}\frac{(-1)^j}{j!}T^j \gamma_1^j T(w')\\
&=\sum_{j\geq 0}\frac{(-1)^j}{j!}T^{j+1} \gamma_1^j T(w')+\sum_{j\geq 0}\frac{(-1)^j}{j!}T^{j}[ \gamma_1^j ,T](w')=0\,.
\end{align*}
By Proposition \ref{prop: injectivityVA} it follows that $v'=0$, and hence $\overline v=0$. 
\end{proof}

\begin{remark}
    The proofs of Propositions \ref{prop: injectivityVA} and \ref{prop: injectivity bracket} work verbatim in the slightly more general setting of lattice vertex algebras with an odd part, as described in \cite[Theorem 3.5]{BLM}.
\end{remark}

\subsection{An application}

We specialize to a very particular example, as in~\cite[Lemma 4.8]{BLM}.

We suppose that $X$ is as above, that $L= K(X)$, and that $Q$ is \textit{the symmetrized pairing} on $L$. That is, for $$\chi(v, w) = \int_X \operatorname{ch}(v^{\vee})\operatorname{ch}(w)\operatorname{Td}(X),$$ the form $Q(v, w)$ is given by $\chi(v, w) + \chi(w, v).$

One immediately gets the lattice vertex algebra $\mathbf V(L)$. 
It is an evident corollary of a theorem of Gross (see~\cite[Lemma 4.8]{BLM}) that Joyce's vertex algebra $H_*(\mathcal M_X)$ (defined in~\cite{Joy})  naturally coincides with $\mathbf V(L)$, since $K(X) = K^{sst}(X)$ by the results of \cite{Voi}.

More precisely, the dual vector space to Joyce's vertex algebra can be identified (via \textit{geometric realization} procedure, see~\cite{BLM}) with the space of descendants $\mathbb D$, so that the pairing between $\mathbf V$ and $\mathbf V^*$ becomes the integration of the descendants against the given homology class.

Unravelling the definitions, one immediately sees that the cohomological map $\tau \to \tau'$ on algebraic descendants from Subsection~\ref{JSw}, for $u$ corresponding to the fundamental class of the point which is the rigidified moduli space of sheaves with the topological type of $\mathcal O(-n)[1]$, is precisely the dual to the Lie bracket with $u$. 

Thus, if we knew that $Q(v_n, \gamma) \neq 0$, $Q(\mathcal O(-n)[1]) = 0$ for some $\gamma$ (equivalently, $Q(\mathbf v, \gamma) \neq 0$, $Q(\mathcal O(-n)[1], \gamma ) = 0$) and for some sufficiently large $n$, we would immediately establish the statement needed for the goal from Subsection~\ref{goal2}. Indeed, Proposition \ref{prop: injectivity bracket} implies it since the condition on $d$ holds, for $n >> 0$, by asymptotic Riemann-Roch theorem.

But this desired statement is an exercise in linear algebra. For an algebraic class $A \in K(X)$, let us denote by $A_i$ the element $\operatorname{ch}_i(A) \in H^{2i}(X)$, and by $\operatorname{td}_i$ the $2i$-th component of $\operatorname{td}_X$. Then,

$$Q(A, B)/2 = A_0B_0 + A_0B_2\operatorname{td}_1 - A_1B_1\operatorname{td}_1 + A_2B_0\operatorname{td}_1.$$

It is evident that one can choose $n$ high enough for the $\gamma$ in question to exist.

Indeed, consider $\widetilde{A}=(A_0, \ A_1,\ A_2)$ (in degrees $0, \ 2, \ 4$) for $A = \textbf{v}$. Choose an integer $n\geq 0$ such that
\[
A_1+A_0\,nH \neq 0 \qquad \text{in } N^1(X)_{\mathbb{Q}}.
\]

Now pick a divisor class $D\in N^1(X)_{\mathbb{Q}}$ such that simultaneously
\[
\int_X (A_1+A_0nH)\,D\,\mathrm{td}_1 \neq 0
\qquad\text{and}\qquad
\int_X H\,D\,\mathrm{td}_1 \neq 0.
\]
This is possible because each condition cuts out a (rational) hyperplane in $N^1(X)_{\mathbb{Q}}$,
and nondegeneracy ensures these hyperplanes are proper.

Let
\[
C := \int_X H^2\,\mathrm{td}_1 \ (\neq 0),
\qquad
\alpha := -n\cdot \frac{\int_X H\,D\,\mathrm{td}_1}{C}\,H^2
\in H^4(X) \otimes \mathbb{Q}.
\]
Finally set
\[
h := (0,D,\alpha).
\]

Then,

\[
\begin{aligned}
Q\bigl(\mathrm{ch}(\mathcal O(-n)),h\bigr)/2
&= \int_X (\alpha + nHD)\,\mathrm{td}_1 \\
&= \int_X \alpha\,\mathrm{td}_1 \;+\; n\int_X H D\,\mathrm{td}_1 \\
&= -n\,\frac{\int_X H D\,\mathrm{td}_1}{C}\int_X H^2\,\mathrm{td}_1 \;+\; n\int_X H D\,\mathrm{td}_1 \\
&= 0.
\end{aligned}
\]

And

\[
\begin{aligned}
Q(\mathbf v,h)/2
&= \int_X (A_0\alpha - A_1D)\,\mathrm{td}_1 \\
&= A_0\int_X \alpha\,\mathrm{td}_1 \;-\; \int_X A_1D\,\mathrm{td}_1 \\
&= -n A_0\int_X H D\,\mathrm{td}_1 \;-\; \int_X A_1D\,\mathrm{td}_1 \\
&= -\int_X (A_1+A_0nH)\,D\,\mathrm{td}_1 \neq 0.
\end{aligned}
\]

\bigskip
\noindent\textsc{Ivan Karpov}\\
Massachusetts Institute of Technology, Department of Mathematics\\
\textit{Email address}: \href{mailto:karpov57@mit.edu}{\texttt{karpov57@mit.edu}}

\medskip
\noindent\textsc{Miguel Moreira}\\
Massachusetts Institute of Technology, Department of Mathematics\\
\textit{Email address}: \href{mailto:miguel73@mit.edu}{\texttt{miguel73@mit.edu}}

\end{document}